\documentclass[11pt, a4paper, oneside, reqno]{amsart}

\usepackage{xcolor}

\definecolor{darkblue}{rgb}{0.0, 0.0, 0.45}
\definecolor{lightblue}{rgb}{0.94, 0.97, 1.0}  
\definecolor{lightblue2}{rgb}{0.68, 0.85, 0.9}
\definecolor{lightcyan}{rgb}{0.88, 1.0, 1.0}
\definecolor{palepink}{rgb}{0.98, 0.85, 0.87}
\definecolor{Mahogany}{rgb}{0.75, 0.25, 0.0} 
\definecolor{ForestGreen}{rgb}{0.13, 0.55, 0.13}  

\usepackage[colorlinks	= true,
raiselinks	= true,
linkcolor	= darkblue, 
citecolor	= Mahogany,  
urlcolor	= ForestGreen,
pdfauthor	= {Mohsen Sadr},
pdftitle	= {},
pdfkeywords	= {},
pdfsubject	= {},
plainpages	= false]{hyperref}

\usepackage{amsmath, amsthm, amssymb, amsfonts}
\usepackage{mathtools, mathrsfs}
\usepackage{enumerate, enumitem}
\usepackage{dsfont}
\usepackage[amssymb, thickqspace]{SIunits}
\usepackage{fancyhdr,mdframed,nicefrac}
\usepackage{epsfig}
\usepackage{graphicx}
\usepackage{float}
\usepackage{caption}
\usepackage{subcaption}
\usepackage{courier}
\usepackage{multirow}
\usepackage{bigstrut}

\allowdisplaybreaks
\date{\today}
\makeatletter
\def\@settitle{\begin{center}%
		\baselineskip14\p@\relax
		\normalfont\LARGE\scshape\bfseries
		\@title
	\end{center}%
}

\def\@setauthors{%
  \begingroup
  \def\thanks{\protect\thanks@warning}%
  \trivlist
  \centering\footnotesize \@topsep30\p@\relax
  \advance\@topsep by -\baselineskip
  \item\relax
  \author@andify\authors
  \def\\{\protect\linebreak}%
  \authors%
  \ifx\@empty\contribs
  \else
    ,\penalty-3 \space \@setcontribs
    \@closetoccontribs
  \fi
  \endtrivlist
  \endgroup
}

\makeatother
\makeatletter

\def\subsection{\@startsection{subsection}{2}%
	\z@{.5\linespacing\@plus.7\linespacing}{.5\linespacing}%
	{\normalfont\large\bfseries}}

\def\subsubsection{\@startsection{subsubsection}{3}%
	\z@{.5\linespacing\@plus.7\linespacing}{.5\linespacing}%
	{\normalfont\itshape}}

\usepackage{multirow}
\usepackage{bigstrut}
\usepackage{wrapfig}
\usepackage{caption}

\newtheorem{theorem}{Theorem}[section]

\newtheorem{lemma}[theorem]{Lemma}

\newtheorem{remark}[theorem]{Remark}

\newtheorem{proposition}[theorem]{Proposition}

\newtheorem{assumption}[theorem]{Assumption}

\newcommand{\R}{\mathbb{R}}

\newcommand{\PS}{\mathcal{P}_2}
\newcommand{\PO}{\mathsf{P}}
\newcommand{\SO}{\mathsf{S}}

\newcommand{\eps}{\varepsilon}

\usepackage[utf8]{inputenc} 
\usepackage[T1]{fontenc}    
\usepackage{hyperref}       
\usepackage{url}            
\usepackage{booktabs}       
\usepackage{amsfonts}       
\usepackage{nicefrac}       
\usepackage{microtype}      
\usepackage{amsmath}

\DeclareMathOperator*{\argmin}{arg\,min}
\usepackage{mathrsfs}  
\usepackage{bm}  
\usepackage{comment}
\usepackage{mathtools}
\usepackage{cleveref}
\usepackage{placeins}
\usepackage{amsopn}
\DeclareMathOperator{\diag}{diag}

\usepackage{algorithm}
\usepackage{algpseudocode}

\newcommand{\defeq}{\mathrel{\mathop:}=}

\title{Optimal Transportation by \\ Orthogonal Coupling Dynamics}
\author{Mohsen Sadr, Peyman Mohajerin Esfahani and Hossein Gorji}
\thanks{Corresponding author: Hossein Gorji. \\
Emails: \href{mailto:msadr@mit.edu}{\texttt{msadr@mit.edu}}, 
\href{mailto:p.mohajerinesfahani@utoronto.ca}{\texttt{p.mohajerinesfahani@utoronto.ca}}, 
\href{mailto:Mohammadhossein.Gorji@empa.ch}{\texttt{Mohammadhossein.Gorji@empa.ch}}.
Mohsen Sadr: Massachusetts Institute of Technology, Cambridge, USA, and Paul Scherrer Institute, Villigen, Switzerland. Peyman Mohajerin Esfahani: University of Toronto, Canada, and Delft University of Technology, the Netherlands. Hossein Gorji: Laboratory for Computational Engineering, Empa, D\"{u}bendorf, Switzerland.}

\usepackage{xcolor}
 \definecolor{ms}{rgb}{0.0, 0.0, 1.}
 \definecolor{hg}{rgb}{1., 0.0, 0.}

\begin{document}
\maketitle

\begin{abstract}
Many numerical and learning algorithms rely on the solution of the Monge-Kantorovich problem and Wasserstein distances, which provide appropriate distributional metrics. While the natural approach is to treat the problem as an infinite-dimensional linear programming, such a methodology limits the computational performance due to the polynomial scaling with respect to the sample size along with intensive memory requirements. We propose a novel alternative framework to address the Monge-Kantorovich problem based on a projection type gradient descent scheme. The dynamics builds on the notion of the conditional expectation, where the connection with the opinion dynamics is leveraged to devise efficient numerical schemes. We demonstrate that the resulting dynamics recovers random maps with favourable computational performance. Along with the theoretical insight, the proposed dynamics paves the way for innovative approaches to construct numerical schemes for computing optimal transport maps as well as Wasserstein distances. 
\end{abstract}


\section{Introduction}
\subsection{Background}
\noindent The importance of the optimal transport lies in its widespread application at different fronts of computational studies, along with its unique theoretical properties. It plays an essential role in machine-learning by offering relevant metrics for comparing different distributions \cite{arjovsky2017wasserstein,kuhn2019wasserstein,dakhmouche2024robust}. It has led to a significant progress in the theory of partial differential equations by its by-product, Wasserstein gradient flows \cite{ambrosio2008gradient,jordan1998variational,santambrogio2017euclidean}. 
At the same time, existing algorithms to compute such distances and corresponding maps remain expensive and complex.
Given the versatile role of optimal transport in different branches of machine-learning \cite{gordaliza2019obtaining,alvarez2018structured,arjovsky2017wasserstein}, analysis \cite{villani2009optimal,jordan1998variational,santambrogio2015optimal}, density functional theory \cite{buttazzo2012optimal,cotar2013density}, optimization \cite{mohajerin2018data}, and inference \cite{goldfeld2024statistical}, among others, several numerical approaches have been pursued for its efficient computations. Classical methods were developed based on linear programming \cite{bonneel2011displacement}, gradient descent \cite{chartrand2009gradient}, dynamic flow formulation \cite{benamou2000computational, trigila2016data}, or elliptic solvers \cite{saumier2015efficient}. For efficient computation and improved performance-scaling with respect to the dimension, entropy regularization \cite{cuturi2013sinkhorn,peyre2019computational} and its stochastic treatment \cite{ballu2020stochastic} have been pursued. Approximate approaches using moment formulation \cite{mula2024moment,sadr2024wasserstein} of Monge-Kantorovich problem have been proposed, as well.  
\\ \ \\
Despite these progresses, to the best of our knowledge, there is no concise Ordinary Differential Equation (ODE), in the sense of particle/agent based dynamics, which leads to the solution of the (un-regularized) Monge-Kantorovich problem. Such a model would be particularly valuable, beyond theoretical and conceptual contributions, for its potential to further enable efficient optimal transport computations.  In the opening chapter of Villani's seminal book \cite{villani2009optimal}, a list of coupling dynamics, of ODE or Stochastic Differential Equation (SDE) forms, are reviewed with the intention of creating  correlations between two given random variables. While interesting results can be obtained e.g. by the Knoth-Rosenblatt rearrangement \cite{el2012bayesian} or Moser's coupling \cite{brenier2003extended}, such schemes remain limited; e.g. due to  coordinate dependency of the Knoth-Rosenblatt rearrangement, or the constraint of small deviation between the marginals for Moser's coupling.
\\ \ \\
In principle, finding the optimal transport map between two probability spaces is a global optimization in the sense of linear programming on an infinite dimensional space. 
We are confronted with the question whether there exist local dynamic rules that can lead to the optimal transportation between two sample spaces? 
An interesting class of dynamic processes is given by orthogonal dynamics. They naturally arise in projecting Hamiltonian systems onto a sub-space, e.g. by the use of conditional expectation \cite{chorin2000optimal,givon2005existence}. These processes have an interesting feature that the remainder/unresolved portion remains orthogonal to the current state of the resolved variable. 
\\ \ \\
Recent study of Conforti et al. \cite{conforti2023projected} leverages the projection method in the context of Langevin dynamics. Their proposed stochastic dynamics leads to the coupling which converges to the Sinkhorn regularization of the optimal transport problem \cite{cuturi2013sinkhorn}. Such a stochastic dynamics, which is closely linked to the Schrödinger Bridge \cite{gentil2020dynamical,leonard2013survey,chen2021stochastic}, gives rise to the Fokker-Planck equation, evolving the joint distribution. It remains to be addressed what happens to the introduced projected Langevin dynamics when the contribution of Brownian motion vanishes, relevant for the (un-regularized) Monge-Kantorovich problem. 
\\ \ \\ 
The simplest form of the orthogonal dynamics can be constructed by projection of the gradient descent (with respect to the Monge-Kantorovich cost) onto a sub-space where marginals are preserved. In the absence of entropy regularization,  this dynamics becomes a good candidate to create optimal transport between two measure spaces. At the level of distribution, one expects such a dynamics gives rise to the Vlasov equation for the evolution of the joint.    
Despite the theoretical appeal of the coupling through projection, numerical treatment of the conditional expectation can become prohibitive. In general, conditional expectation is the solution of the Bayesian regression \cite{nadaraya1964estimating}. The Hegselmann-Krause model of opinion dynamics
 \cite{hegselmann2006truth, li2013consensus} offers non-parametric numerical approach to deal with such coupling dynamics. With these observations in hand and leveraging the recent study \cite{conforti2023projected}, we pursue a projected dynamic approach as a concise solution algorithm of the Monge-Kantorovich problem.   
\subsection{Main Contributions}
We propose a numerical scheme for the Monge-Kantorovich problem built on the following steps.
\begin{enumerate}
\item {\bf Coupling Dynamics.} We introduce the dynamics in Eq.~\eqref{eq:main-ODE} of \cref{sec:main-dynamics}, motivated by projection of the gradient descent on a marginal-preserving tangent space.
\item {\bf General Structural Properties.} In \cref{sec:struct-prop}, assuming that the well-posedness of the dynamics is granted, we justify several structural properties for the general cost setting.
\begin{enumerate}
    \item {\it Marginal Preservation.} We prove in  Proposition~\ref{prop:marg-pres} that the introduced dynamics preserves the marginals.
    \item {\it Descent in Cost.} The dynamics has a monotone behaviour in decreasing the cost, as shown in {Proposition~\ref{prop:pos}}.
    \item {\it Instability of Sub-optimal Couplings.} We find that the sub-optimal maps are unstable as proved in Proposition~\ref{prop:sub-opt}.
    \item {\it McKean-Vlasov System.} We derive the corresponding kinetic description of the introduced evolution in {Theorem}~\ref{theorem:vlasov}, where the kinetics takes the form of the Vlasov equation \eqref{eq:vlasov}. 
\end{enumerate}
\item {\bf Key Features in $L^2$-Monge-Kantorovich.} Next, we turn our attention to the common case of $L^2$-cost and derive two results in \cref{sec:key-L2}.
\begin{enumerate}
\item {\it Symmetric Positive Semi-Definite Correlation.} We observe that the dynamics keeps the centered cross-correlation between samples of the two measure spaces in the cone of symmetric positive semi-definite matrices, as proved in Proposition \ref{prop:spd}.
\item {\it Sharp Descent and Variational Formulation.} We derive the variational form of the dynamics, which according to Theorem \ref{theorem:sharp-descent} gives the sharpest descent in the $L^2$-Monge-Kantorovich cost among a class of marginal-preserving dynamics.
\end{enumerate}
\item 
{\bf Characterization in Gaussian Setting.} In \cref{sec:special_examples}, we study the Riccati dynamics governing the cross-correlation of Gaussian marginals under $L^2$-OCD and prove in Proposition 
\ref{prop:OCD-Gaussian}, that the algorithm converges to the true OT if the marginal covariances commute. We further justify that the convergence is exponentially fast in Theorem \ref{thm:expon-gauss}.
\item {\bf Numerical Algorithm.} In \cref{sec:numerics}, after motivating the analogy with opinion dynamics, a non-parametric recipe is proposed for computing the conditional expectation. 
\item{\bf Numerical Experiments.} We perform several numerical experiments to verify the convergence of the dynamics, as well as its relevance in practical settings, in \cref{sec:results}. We demonstrate that the presented dynamics, equipped with the devised numerical algorithm, successfully recover nonlinear Monge maps. Furthermore, we showcase its performance in distribution learning and color interpolation. 
\end{enumerate}
Our focus remains on basic structural properties of the dynamics along with its computational implications. We demonstrate the relevance of the orthogonal coupling dynamics and leave more thorough theoretical and numerical analysis of the dynamics for future studies.
\subsection{Notation and Setup}
\noindent Let $\PS(\Omega)$ be the set of Borel measures over $\Omega$ with a finite second moment, i.e.  $\PS (\Omega)\defeq \{\mu\in\mathcal{P}(\Omega): \int_{\Omega} \lVert x \rVert_2^2 \mu(dx)<\infty\} $, where $\lVert . \rVert_2^2$ is the usual $L^2$ Euclidean norm, and for $f\in L^2(\mu;\mathbb{R}^n)$ let
\begin{eqnarray}
\|f\|_{L^2(\mu)}
&\defeq&\left( \int_{\mathbb{R}^n} \|f(x)\|_2^2\,\mu(dx) \right)^{1/2}
\end{eqnarray} 
be the $L^2$ norm with respect to the measure $\mu$. We write $\mu\otimes\nu$ for the product measure on 
$\mathbb{R}^n\times\mathbb{R}^n$. 
For vectors $A,B\in\mathbb{R}^n$, the symbol $A\otimes B$ denotes the 
rank--one matrix $AB^{\top}$. \noindent Suppose $\Pi(\mu,\nu)$ is the set of all joint measures in $\PS(\R^{n}\times \R^{n})$ with marginals $\mu,\nu\in \PS(\R^n)$. 
We denote by $\pi^{\textrm{opt}}$ the solution of the Monge-Kantorovich problem
\begin{eqnarray}
\label{eq:pi_opt}
\pi^{\textrm{opt}}_{\mu,\nu}&\defeq&\argmin_{\pi\in\Pi(\mu,\nu)}\int_{\R^{n}\times \R^{n}} c(x,y)\pi(dx,dy)
\end{eqnarray}
with $c:\mathbb{R}^{2n}\to\mathbb{R}$ as a non-negative smooth function (see e.g. \cite{bogachev2012monge} for details).
For the standard setting of $c(x,y)=\frac{1}{2}\lVert x-y \rVert_2^2$,
the optimal map $\hat{T}:\R^{n}\to\R^{n}$ exists, if $\mu$ and $\nu$ are absolutely continuous with respect to the Lebesgue measure \cite{ambrosio2008gradient}.  It gives $\hat{T}\# \mu=\nu$ ($\#$ is the push-forward) and accordingly the 2-Wasserstein distance reads
\begin{eqnarray}
d^2(\mu,\nu)&=& \frac{1}{2} \int_{\R^{n}\times \R^{n}} \lVert x-y \rVert_2^2\pi^{\textrm{opt}}_{\mu,\nu}(dx,dy) \\
&=& \frac{1}{2} \int_{\R^{n}} \lVert \hat{T}(x)-x \rVert_2^2\mu(dx) \ .
\end{eqnarray}
\sloppy Let $\mathcal{H}=\mathcal{L}^2(\Omega,\mathscr{A},P)$ be the Hilbert space comprised of square-integrable $\mathscr{A}$-measurable functions on $\Omega$. Suppose $X_t,Y_t:\Omega\to\R^n$ are random processes in $\mathcal{H}$ with the index $t\in[0,+\infty)$ and joint law $p_t$. We employ the conditional expectation $\mathbb{E}_{p_t}[X_t|Y_t]$ as the orthogonal projection of $X_t$ onto the span of $\sigma(Y_t)$-measurable functions in $\mathcal{H}$, denoted by $\PO_{{p_t},{Y_t}}[X_t]$, where $\sigma(A)$ is the smallest $\sigma$-algebra with respect to which the random variable $A$ is measurable (see e.g. \cite{chung2000course,borovkov1999probability} for further details). 
Accordingly, we have
\begin{eqnarray}
\PO_{p_t,X_t}[Y_t]=\mathbb{E}_{p_t}[Y_t|X_t] \ \ \ \textrm{and} \ \ \ \PO_{p_t,Y_t}[X_t]=\mathbb{E}_{p_t}[X_t|Y_t] \ .
\end{eqnarray}
We consider the probability flows following the continuity equation of the form  
\begin{eqnarray}
\label{eq:p-flow}
\partial_tp_t+\nabla \cdot (vp_t)&=&0 \ ,
\end{eqnarray}
initialized by $p_0\in \Pi(\mu,\nu)$, 
where the velocity vector $v=(v_1,v_2)$ is a functional of $p_t$ and belongs to the tangent space \cite{ambrosio2007gradient}
\begin{eqnarray}
\textrm{Tan}_{p_t} \mathcal{P}_2(\mathbb{R}^{2n})&\defeq &\overline{\left\{\nabla h : h \in C_c^\infty(\mathbb{R}^{2n})\right\}}^{L^2(p_t;\mathbb{R}^{2n})} \ .
\end{eqnarray}
Note that $\textrm{Tan}_{p_t} \mathcal{P}_2(\mathbb{R}^{2n})$ is a Hilbert space comprised of gradient vector fields. 
Besides $\textrm{Tan}_{p_t} \mathcal{P}_2(\mathbb{R}^{2n})$, we specifically consider two tangent spaces $\textrm{Tan}_{p_t,1}\Pi(\mu,\nu)$ and $\textrm{Tan}_{p_t,2}\Pi(\mu,\nu)$. The former, following \cite{conforti2023projected}, reads
\begin{eqnarray}
\label{eq:tan-pres}
\textrm{Tan}_{p_t,1}\Pi(\mu,\nu)
\defeq\left\{v\in \textrm{Tan}_{p_t} \mathcal{P}_2(\mathbb{R}^{2n}) : \int_{\mathbb{R}^{2n}} v \cdot \begin{pmatrix} \nabla h_1 \\ \nabla h_2 \end{pmatrix}  dp_t=0, \forall h_{1,2} \in C_c^\infty (\mathbb{R}^n)\right\} \ , \nonumber \\
\end{eqnarray}
which gives the set of all velocities that preserve the marginals $\mu$ and $\nu$. Following notation of \cite{conforti2023projected}, the latter tangent space is given by
\begin{eqnarray}
\label{eq:tan-simple}
\textrm{Tan}_{p_t,2}\Pi(\mu,\nu)&\defeq &\left\{v\in \textrm{Tan}_{p_t} \mathcal{P}_2(\mathbb{R}^{2n}) : \begin{pmatrix} \mathbb{E}_{p_t}[v_1(X,Y)|X] \\
\mathbb{E}_{p_t}[v_2(X,Y)|Y]
\end{pmatrix} =0, \ a.s.
\right\},
\end{eqnarray}
\noindent which is a sub-space of $\textrm{Tan}_{p_t,1}\Pi(\mu,\nu)$. We denote it by $\ker(\PO_{p_t,X})\times\ker( \PO_{p_t,Y})$ for simplicity and as shown in Proposition \ref{prop:marg-pres}, it preserves the marginals $\mu$ and $\nu$. Since for $v\in\textrm{Tan}_{p_t,2}\Pi(\mu,\nu)$, we have $\mathbb{E}_{p_t}[v_1(X_t,Y_t)|X_t]=\mathbb{E}_{p_t}[v_2(X_t,Y_t)|Y_t]=0$,  the projection operator  $\SO_{p_t}: \textrm{Tan}_{p_t} \mathcal{P}_2(\mathbb{R}^{2n})\to \textrm{Tan}_{p_t,2}\Pi(\mu,\nu)$ 
is orthogonal to $\PO_{p_t}$, i.e.
\begin{eqnarray}
\SO_{p_t,X_t}[v_1(X_t,Y_t)]&=&v_1(X_t,Y_t)-\mathbb{E}_{p_t}[v_1(X_t,Y_t)|X_t] \\ 
\textrm{and} \ \ \ \SO_{p_t,Y_t}[v_2(X_t,Y_t)]&=&v_2(X_t,Y_t)-\mathbb{E}_{p_t}[v_2(X_t,Y_t)|Y_t] \ ,
\end{eqnarray}
for $v\in \textrm{Tan}_{p_t}\mathcal{P}_2(\mathbb{R}^{2n})$. 

\section{Orthogonal Coupling} 
\subsection {Main Dynamics}
\label{sec:main-dynamics}
\noindent Starting with independent samples of $X_0\sim \mu$ and $Y_0\sim \nu$, we would like to construct a micro-dynamics in the sense of an ODE on $X_t$ and $Y_t$, 
which converges to the solution of \eqref{eq:pi_opt}. In other words, we look for a velocity vector $v$ which updates $X_t$ and $Y_t$ in a way that the minimum cost $c(X_t,Y_t)$, in the average sense, is attained as $t\to \infty$. One can consider the conventional gradient descent for $X_t$ and $Y_t$ with respect to $c$, suggesting $v_1^{\textrm{gd}}=-\nabla_x c$ and $v_2^{\textrm{gd}}=-\nabla_y c$, and updating $X_t$ and $Y_t$ according to $v_1^{\textrm{gd}}$ and $v_2^{\textrm{gd}}$, respectively. However, it is easy to spot that the distributions of $X_t$ and $Y_t$ would go out of the set of couplings $\Pi (\mu,\nu)$. \\ \ \\
An interesting approach to keep the marginals in $\Pi (\mu,\nu)$ is to project the dynamics of the gradient descent onto the tangent space which preserves the marginals. In particular, if we consider the probability flow given by
\begin{eqnarray}
\partial_t p_t+\nabla\cdot(\hat{v}p_t)&=&0 \ ,
\end{eqnarray}
and restrict the velocities to those $\hat{v}=(\hat{v}_1,\hat{v}_2)\in \textrm{Tan}_{p_t} \PS(\mathbb{R}^{2n})$ which fulfill 
\begin{eqnarray}
\int_{\mathbb{R}^{2n}} \left(\hat{v}_1 \cdot \nabla h_1(x)+\hat{v}_2\cdot \nabla h_2(y)\right) p_t(dx,dy)&=&0 \, \ \ \ \forall h_{1,2}\in C^{\infty}(\mathbb{R}^n),
\end{eqnarray}
the marginals of $p_t$ remain unchanged.
However, the general form of this tangent space, given by $\textrm{Tan}_{p_t,1} \Pi(\mu,\nu)$ in Eq.~\eqref{eq:tan-pres}, does not provide us with a concise and convenient projection (this avenue was explored, to some degree, in \cite{angenent2003minimizing}). Yet, if instead of $\textrm{Tan}_{p_t,1} \Pi(\mu,\nu)$, we focus on a sub-manifold, which still preserves the marginals, the problem can become tractable. In particular, the tangent space $\textrm{Tan}_{p_t,2} \Pi(\mu,\nu)=\ker(\PO_{p_t,X})\times\ker( \PO_{p_t,Y})$ offers a straight-forward projection, while preserving the marginals. Noting that the Riemannian gradient is the orthogonal projection of the Euclidean one onto the tangent space \cite{boumal2023introduction}, we get the projected version of $v^{\textrm{gd}}$, as shown in Fig. \ref{fig:pgd}.
\begin{figure}
  	\centering
\includegraphics[scale=0.03]{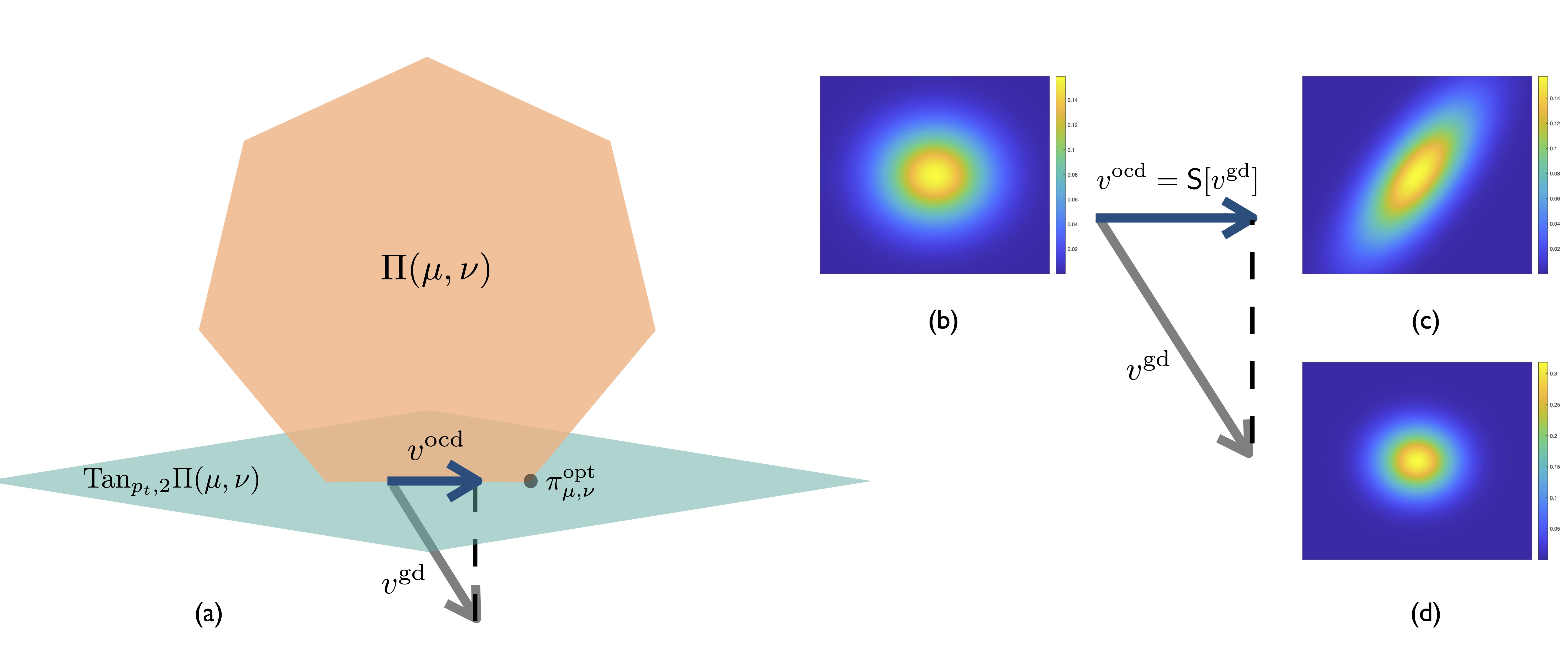}
\caption{Geometric illustration of OCD; (a) The gradient descent, which leaves the polytope $\Pi(\mu,\nu)$, is projected onto the tangent space $\ker(\PO_{p_t,X})\times\ker(\PO_{p_t,Y})$; (b) Initial condition as product Gaussian probability density; (c) Evolved by $v^{\textrm{ocd}}$, the initial density is updated by projected gradient descent in $L^2$-cost; (d) Evolved by $v^{\textrm{gd}}$, the initial density is updated by gradient descent in $L^2$-cost, leaving $\Pi(\mu,\nu)$.  }
   \label{fig:pgd}

\end{figure}
In fact, projecting $v^{gd}$ onto $\textrm{Tan}_{p_t,2} \Pi(\mu,\nu)$ gives us the update rule 
\begin{eqnarray}
\dot{X}_t=v_1^{\textrm{ocd}}\defeq-\nabla_x c(X_t,Y_t)+\mathbb{E}_{p_t}[\nabla_x c(X_t,Y_t)|X_t] \nonumber \\
\textrm{and} \ \ \ \ \dot{Y}_t=v_2^{\textrm{ocd}}\defeq-\nabla_y c(X_t,Y_t)+\mathbb{E}_{p_t}[\nabla_y c(X_t,Y_t)|Y_t] \,
\label{eq:main-ODE}
\end{eqnarray}
which we refer to as the Orthogonal Coupling Dynamics (OCD). The orthogonality comes from the observation that $\mathbb{E}[v_1^{\textrm{ocd}}\cdot X_t]=\mathbb{E}[v_2^{\textrm{ocd}}\cdot Y_t]=0$. In general, there is no guarantee that OCD reaches global minimum of \eqref{eq:pi_opt}, as the projection onto the sub-manifold $\textrm{Tan}_{p_t,2}\Pi(\mu,\nu)$ can prevent the dynamics from visiting all possible couplings between $\mu$ and $\nu$. This can be readily noticed, e.g. if $Y_0$ is initialized as a sub-optimal map of $X_0$. As a result, the dynamics remains frozen at a sub-optimal fixed point. However, as we will discuss, there are several nuances in \eqref{eq:main-ODE}, which warrant investigation of OCD on both theoretical and numerical fronts. In particular, we observe that OCD ensures the cost is non-increasing. This result plays a crucial role in two important properties that we deduce: we observe that sub-optimal couplings turn out to be unstable (see Proposition \ref{prop:sub-opt}), and that OCD gives the sharpest decay in $L^2$-cost among processes with velocity in $\textrm{Tan}_{p_t,2} \Pi (\mu,\nu)$ (see Theorem \ref{theorem:sharp-descent}). In practice, the dynamics reaches accurate estimations of the optimization problem \eqref{eq:pi_opt}, as we demonstrate in results section. \\ \ \\
From conceptual point of view, by replacing the optimization problem  \eqref{eq:pi_opt} with the OCD system \eqref{eq:main-ODE}, we turn the Monge-Kantorovich problem into a series of regression problems, implied by the conditional expectation. While the original problem, given by Eq.~\eqref{eq:pi_opt}, is a constrained optimization over distributions, the OCD entails an unconstrained optimization over functions, in $L^2$. This transformation allows us to leverage tools of conditional expectation estimation in order to achieve efficient computation of optimal transport in contrast to the original constrained infinite dimensional linear programming problem. \\ \ \\
However, before analyzing OCD, let us mention that alternatively, one can consider the Langevin dynamics 
\begin{eqnarray}
\label{eq:SDE}
d{X}_t&=&\left(-\nabla_x c+\mathbb{E}_{p_t}[\nabla_x c|X_t]+\epsilon \nabla_x \log f_\mu(X_t)\right)dt+\sqrt{2\epsilon} dW_t \nonumber \\
\textrm{and} \ \ \ \ d{Y}_t&=&\left(-\nabla_y c+\mathbb{E}_{p_t}[\nabla_y c|Y_t]+\epsilon \nabla_y \log f_\nu(Y_t)\right)dt+\sqrt{2\epsilon} dB_t \ ,
\end{eqnarray}
where $\epsilon >0$ is the regularization parameter; $W_t,B_t$ are independent Wiener processes with values in $\mathbb{R}^n$; and $f_{\mu,\nu}$ are densities of the measures $\mu,\nu$. Assuming log-concave densities, Conforti et al \cite{conforti2023projected}
 have shown that the measure generated by the Langevin dynamics \eqref{eq:SDE} converges to the solution of \eqref{eq:pi_opt}, once the latter is regularized by entropy with parameter $\epsilon$. Despite the theoretical grounding of SDE \eqref{eq:SDE}, we have three main motivations behind focusing on OCD \eqref{eq:main-ODE}, instead:
 \begin{enumerate}
 \item While the SDE system gives us an approximation of the optimal joint density, it would not give us a map between $X_t$ and $Y_t$, due to the presence of the Brownian motion. Therefore, for applications where the random map is of interest, the SDE may not be employed; at least without further treatments.
 \item The densities of the measures may not be known in many relevant applications. We are typically left only with samples of $\mu$ and $\nu$. Therefore the evaluation of the SDE dynamics requires further regularizations. 
 \item From the numerical aspect,  computing the grad-log terms of the density, even if the latter is given, can be significantly costly if $n$ is large in $\mathbb{R}^n$. 
 \end{enumerate}
The open question remains under which technical assumptions the case of $\epsilon=0$, which is OCD \eqref{eq:main-ODE}, would converge to the solution of the Monge-Kantorovich problem. While we do not directly address this problem and despite this knowledge gap, we present new results on both theoretical and numerical aspects of OCD, presenting its basic features along with computational relevance for practical optimal transport problems. 

\subsection{General Structural Properties}
\label{sec:struct-prop}
\noindent The OCD dynamics \eqref{eq:main-ODE} leads to a certain favourable theoretical features. These are rooted in the fact that the conditional expectation is a projection operator in $L^2$ with self-adjoint and contraction characteristics. We leverage these characteristics to derive theoretical properties of the OCD. We provide an overview of the main theoretical results below, where the proofs are given in \cref{sec:proof}. 
\begin{enumerate}
\item {\bf Marginal Preservation.} Since the velocities $v_{1,2}^{\textrm{ocd}}$, introduced in Eq.~\eqref{eq:main-ODE}, are orthogonal with respect to the functions of $X_t$ and $Y_t$, respectively, a direct computation shows that the distributions of $X_t$ and $Y_t$ remain unchanged. More precisely we have the following result.
\begin{proposition}
\label{prop:marg-pres}
Consider $(X_t,Y_t)$ to be the solution of the ODE system
\begin{eqnarray}
\label{eq:gen-ODE}
\dot{X}_t=v_1 \ \ \ \textrm{and} \ \ \ \dot{Y}_t=v_2
\end{eqnarray}
 with velocities $v=(v_1,v_2)\in \textrm{Tan}_{p_t,2} \Pi(\mu,\nu)$, and the initial condition $(X_0,Y_0)\sim \pi$, where $\pi\in\Pi(\mu,\nu)$ and $p_t$ is their joint law. Under the regularity assumptions stated in 
Assumption~\ref{ass:regularity}, $p_t\in \Pi(\mu,\nu)$.
\end{proposition}
\item {\bf Descent in Cost.} As a result of the conditional expectation contraction property, the OCD leads to a monotone decay in the cost. 
More specifically, we have the following result.
\begin{proposition}{
\label{prop:pos}
Consider $(X_t,Y_t)$ to be the solution of OCD \eqref{eq:main-ODE}  with the initial condition $(X_0,Y_0)\sim \pi$, from an arbitrary coupling $\pi \in \Pi(\mu,\nu)$. Under the regularity assumptions stated in 
Assumption~\ref{ass:regularity}, we have 
\begin{eqnarray}
\label{eq:corr-pos}
\frac{d}{d t}\mathbb{E}[c(X_t,Y_t)] \le 0 \ .
\end{eqnarray}
}
\end{proposition}
\noindent In other words, the OCD \eqref{eq:main-ODE} monotonically reduces the cost $\mathbb{E}[c(X_t,Y_t)]$ until the latter reaches its  stationary value.  This property is crucial to show that the sub-optimal fixed points of the dynamics are unstable (as discussed in the following). 
\item {\bf Instability of Sub-Optimal Couplings.} The monotonic decay in the cost allows us to investigate stability of sub-optimal couplings. In particular, we have that the solution $(X_t,Y_t)$ of OCD with a measure from the sub-optimal coupling set
\begin{eqnarray}
\Pi^*(\mu,\nu)\defeq \left\{\pi\in \Pi(\mu,\nu):  \int_{\mathbb{R}^{2n}}c(x,y)\pi(dx,dy)>\int_{\mathbb{R}^{2n}}c(x,y)\pi^{\textrm{opt}}(dx,dy)\right\} \nonumber
\end{eqnarray}
 is unstable, in the sense described below.
 \begin{proposition}\label{prop:sub-opt}{ Consider $(X_t,Y_t)$ to be the solution of the ODE system \eqref{eq:main-ODE}  with the initial condition $(X_0,Y_0)\sim \pi$, from an arbitrary coupling $\pi \in \Pi(\mu,\nu)$. Suppose that the joint measure of $(X_t,Y_t)$ fulfills
\begin{eqnarray}
p_T\in \Pi^*(\mu,\nu) 
\end{eqnarray}
at some $T>0$.
 Under the regularity assumptions stated in 
Assumption~\ref{ass:regularity}, the coupling $p_T$ is unstable in the sense that if the law is perturbed by
\begin{eqnarray}
p_T^\epsilon&=&(1-\epsilon)p_T+\epsilon \pi^{\textrm{opt}}_{\mu,\nu} ,
\end{eqnarray}
where $0<\epsilon<1$, the solution $(X_t,Y_t)$ would not have the measure $p_T$ for any time later. 

}
\end{proposition}
\item {\bf McKean-Vlasov System.}
The OCD \eqref{eq:main-ODE} can be equivalently interpreted as a probability flow in the sense of Eq.~\eqref{eq:p-flow}. Due to the absence of Brownian motion, this probability flow is in the form of the Vlasov equation. The latter arises in the mean field limit of interacting particles, e.g. see \cite{golse2003mean,paul2022microscopic}. By exploiting the projection operator $\SO$, the corresponding probability density evolution takes the concise form
\begin{eqnarray}
\label{eq:vlasov}
\partial_t \rho_t+\nabla_x \cdot \left(\rho_t\SO_{p_t,x}[\nabla_x c]\right)+\nabla_y \cdot \left(\rho_t\SO_{p_t,y}[\nabla_y c]\right)&=&0 \ ,
\end{eqnarray}
where
\begin{eqnarray}
\SO_{p_t,x}[\nabla_x c]&=&\nabla_x c(x,y)-\mathbb{E}_{(X_t,Y_t)\sim p_t}[\nabla_x c(X_t,Y_t)|X_t=x] \nonumber  \\
\textrm{and} \ \ \ \SO_{p_t,y}[\nabla_y c]&=&\nabla_y c(x,y)-\mathbb{E}_{(X_t,Y_t)\sim p_t}[\nabla_y c(X_t,Y_t)|Y_t=y] \ .
\end{eqnarray}
\begin{theorem}
\label{theorem:vlasov}
{Consider $(X_t,Y_t)$ to be the solution of the ODE system \eqref{eq:main-ODE}  with the initial condition $(X_0,Y_0)\sim \pi$, from an arbitrary coupling $\pi \in \Pi(\mu,\nu)$. Suppose $p_t$ is the joint law of $(X_t,Y_t)$ with the density $\rho_t$. Under the regularity assumptions stated in 
Assumption~\ref{ass:regularity}, $\rho_t$ is weak solution of the Vlasov-type equation Eq.~\eqref{eq:vlasov}.
}
\end{theorem}
\noindent We expect that the Vlasov equation \eqref{eq:vlasov} admits a variational formulation in terms of Otto calculus \cite{jordan1998variational,ambrosio2007gradient,santambrogio2017euclidean, bashiri2020gradient}, if appropriate set of couplings is considered (those that can be generated by $\textrm{Tan}_{p_t,2} \Pi(\mu,\nu)$). However, we leave this interesting question to a separate work. Instead, we will present a variational formulation of OCD in the probability space  for the $L^2$ Monge-Kantorovich cost, in the next section. 
\end{enumerate}

\subsection{Key Features in $L^2$ Monge-Kantorovich} 
\label{sec:key-L2}
\noindent So far, we discussed that OCD  \eqref{eq:main-ODE} brings a given joint distribution closer to the solution of the Monge-Kantorovich problem. In other words, the dynamics decreases the cost $\mathbb{E}[c(X_t,Y_t)]$. In this section, we present a finer result on the decay of the quadratic cost $c(x,y)=\lVert x-y \rVert_2^2$, resulted from OCD, i.e. for the  ODE system
\begin{eqnarray}
\dot{X}_t=Y_t-\mathbb{E}_{p_t}[Y_t|X_t] \ \ \ \textrm{and} \ \ \ 
\dot{Y}_t=X_t-\mathbb{E}_{p_t}[X_t|Y_t] \label{eq:l2-ocd} \ .
\end{eqnarray}
In this scenario, the optimal transport problem is equivalent to maximizing correlation between $X_t$ and $Y_t$, see \cite{rachev2006mass,ruschendorf1990characterization}. Focusing on \eqref{eq:l2-ocd}, we show the following results.
\begin{enumerate}
\item {\bf Symmetric Positive Semi-Definite Correlation.} Consider $S^n_{+}$ to be the set of symmetric positive semi-definite matrices
\begin{eqnarray}
S^n_{+}&\defeq&\left\{A\in \mathbb{R}^{n\times n}|A=A^T, A\succeq 0\right\} \ .
\end{eqnarray}
Let us define the cross-correlation matrix
\begin{eqnarray}
J_{t}&\defeq&\mathbb{E}\left[X_t^\prime \otimes {Y_{t}^\prime} \right]
\end{eqnarray}
with $X_t^\prime\defeq X_t-\mathbb{E}[X_t]$ and $Y_t^\prime\defeq Y_t-\mathbb{E}[Y_t]$. We have the following result.
\begin{proposition}{
\label{prop:spd}
Consider $(X_t,Y_t)$ to be the solution of the ODE system \eqref{eq:main-ODE}  with the initial condition $(X_0,Y_0)\sim \mu \otimes \nu$, given the regularity assumptions stated in 
Assumption~\ref{ass:regularity}. Therefore 
\begin{eqnarray}
J_t\in S^n_{+} \ .
\end{eqnarray}
}
\end{proposition}
\noindent In other words the $L^2$-OCD \eqref{eq:l2-ocd}  ensures that the cross-correlation created between $X_t$ and $Y_t$ remains in the space of symmetric positive semi-definite matrices. 
 \item {\textbf{Sharp Descent and Variational Formulation.}}   The $L^2$-OCD \eqref{eq:l2-ocd} gives rise to a sharp  decay in the cost $\mathbb{E}[c(X_t,Y_t)]$. More precisely, if we restrict the set of velocities to the tangent space $ \ker(\PO_{p_t,X})\times\ker( \PO_{p_t,Y})$, the OCD gives fastest decay in the cost  $\mathbb{E}[c(X_t,Y_t)]$.
 \begin{theorem}
\label{theorem:sharp-descent}
 Under the regularity assumptions stated in 
Assumption~\ref{ass:regularity}, consider $(X_t,Y_t)$ to be the solution of the ODE system \eqref{eq:gen-ODE} and the initial condition $(X_0,Y_0)\in \pi$, where $\pi\in\Pi(\mu,\nu)$. 
    Let 
\begin{eqnarray}
\label{eq:potential}
\mathcal{C}(p_t)&\vcentcolon=&\frac{1}{2}\int_{\mathbb{R}^{2n}} \lVert x-y \rVert_2 ^2 p_t (dx,dy)
\end{eqnarray}
be the $L^2$ transport cost and
\begin{eqnarray}
\label{eq:lyup}
\mathcal{L}_{p_t}(v)&\vcentcolon=&\frac{d}{dt}\mathcal{C}(p_t)+\frac{1}{2}\int_{\mathbb{R}^{2n}}\lVert v \rVert_2^2p_t(dx,dy) \ ,
\end{eqnarray}
its decay rate penalized by the $L^2$ norm of the velocities.
Define optimal velocities
\begin{eqnarray}
\label{eq:opt-vel-lyup}
v^{opt}_{p_t}&\defeq&\argmin_{v\in  \ker(\PO_{p_t,X})\times\ker( \PO_{p_t,Y})} \mathcal{L}_{p_t}(v) \ .
\end{eqnarray}
Therefore
\begin{eqnarray}
\label{eq:v-opt}
v_{p_t,1}^{opt}(X_t,Y_t)&=&Y_t-\mathbb{E}_{p_t}[Y_t|X_t]\\
\textrm{and} \ \ \ \ v_{p_t,2}^{opt}(X_t,Y_t)&=&X_t-\mathbb{E}_{p_t}[X_t|Y_t] \label{eq:w-opt} \ .
\end{eqnarray}
\end{theorem}
\noindent This result can be employed as variational formulation of $L^2$-OCD. The velocities $v_{1,2}^{\textrm{ocd}}$ admit the variational formulation
 \begin{eqnarray}
 v_{1,2}^{\textrm{ocd}}=  \argmin_{v\in \ker(\PO_{p_t,X})\times\ker(\PO_{p_t,Y})}\left[\frac{1}{2}\frac{d}{dt} \int_{\mathbb{R}^{2n}}\lVert x-y \rVert_2^2dp_t+\frac{1}{2}\int_{\mathbb{R}^{2n}} \lVert v \rVert_2^2dp_t\right] \ . \nonumber
 \end{eqnarray}

\begin{remark}\label{rem:JKO-1D}
In 1D, we expect that the variational formulation given above can be cast in the Jordan-Kinderlehrer-Otto (JKO) type flow \cite{jordan1998variational}. We do not present a proof but motivate this by the following. Assume $n=1$, $\mu,\nu\in\mathcal P_2(\mathbb R)$ are absolutely continuous. 
Then one can define the JKO form of $L^2$-OCD as
\begin{equation}\label{eq:JKO-1D}
p^{k+1}\in\arg\min_{p\in\Pi(\mu,\nu)}\Big\{\tfrac{1}{2\tau}W_2^2(p,p^k)+\mathcal{C}(p)\Big\}
\end{equation}
with $\mathcal{C}(p)$ given by \eqref{eq:potential}.
We expect the limit curve $p_t$ of \eqref{eq:JKO-1D} as $\tau\downarrow0$ becomes the gradient flow of $\mathcal{C}(p)$ on $\Pi(\mu,\nu)$, with the velocity field $v_t$ given by $L^2$-OCD. This can be justified since in 1D, any marginal preserving scheme should have a velocity field $v_t\in \ker(\PO_{p_t,X})\times\ker( \PO_{p_t,Y})$. However we leave the technical proof of this analogy to dedicated studies. 
\end{remark}
\begin{remark}\label{rem:gen-cost}
Note that one can extend the steep descent formulation to a general continuously 
differentiable cost $c\in C^1(\mathbb R^n\times \mathbb R^n)$. While not pursued here in detail, we give a general argument.
For the functional $ \mathcal{C}(p)=\int c\,dp$, the first variation in a 
marginal-preserving direction $v=(v_1,v_2)\in {\ker(\PO_{p_t,X})\times\ker( \PO_{p_t,Y})}$ is
\begin{eqnarray}
D\mathcal{C}_p[v]
&=&\int \left(\nabla_x c\, v_1 + \nabla_y c\, v_2\,\right) dp.
\end{eqnarray}
Consequently, the unique minimizer of the quadratic functional
\[
v\mapsto D\mathcal{C}_p[v]+\tfrac12\|v\|_{L^2(p)}^2
\quad\text{over}\quad v=(v_1,v_2)\in {\ker(\PO_{p_t,X})\times\ker( \PO_{p_t,Y})}
\]
is the orthogonal projection of $-(\nabla_x c,\nabla_y c)$ onto
${\ker(\PO_{p_t,X})\times\ker( \PO_{p_t,Y})}$, namely
\[
v_1^{\textrm{ocd}}=-\nabla_x c+\mathbb E[\nabla_x c\mid X],\qquad
v_2^{\textrm{ocd}}=-\nabla_y c+\mathbb E[\nabla_y c\mid Y].
\]

\end{remark}

\end{enumerate}
\section{Technical Results}
\label{sec:proof}
\subsection {{Review of Conditional Expectation Properties}}
In order to establish basic theoretical properties of the OCD \eqref{eq:main-ODE}, we often make use of the two following properties of the conditional expectation for random variables $A, B, C \in \mathcal{H}$ \cite{chung2000course, borovkov1999probability}.
\begin{enumerate}
\item \label{prop:cont} {\it Contraction property:} As a result of Jensen's inequality, we have
\begin{eqnarray}
\mathbb{E}[\lVert A\rVert_2^2] \ge \mathbb{E}\left[\lVert \mathbb{E}[A|B] \rVert_2^2\right] \ .
\end{eqnarray}
\item \label{prop:self} {\it Self-adjoint property:}  Due to the law of total expectation, we have
\begin{eqnarray}
\mathbb{E}[A  \cdot \mathbb{E}[B|C]]&=&\mathbb{E}[B  \cdot \mathbb{E}[A|C]] \ .
\end{eqnarray}
\end{enumerate}
\subsection{Assumptions}
 We do not attempt to prove well-posedness of \eqref{eq:main-ODE} in this work, instead we collect below a set of regularity assumptions under which the
subsequent derivations are justified. These are standard in the literature on
McKean--Vlasov dynamics, see e.g.~\cite{sznitman1991topics,bernou2023path}.
In particular, in order to justify the analytical steps, we make explicit the following regularity assumptions.
\begin{assumption}[Regularity of densities and velocities]
\label{ass:regularity}
For each $t\ge 0$, the joint law $p_t$ of $(X_t,Y_t)$ admits a Lebesgue density
$\rho_t \in L^1(\mathbb{R}^{2n})$ with finite second moment. Moreover:
\begin{enumerate}
\item The map $t\mapsto \rho_t$ is weakly continuous in $L^1$.
\item The velocity field $v_t=(v_{1,t},v_{2,t})$ satisfies
$v_t \in L^2(p_t)$ and $v_t \rho_t \in L^1(\mathbb{R}^{2n})$.
\item The cost satisfies $c\in C^1(\mathbb{R}^{2n})$ and its gradients
$\nabla_x c$, $\nabla_y c$ have at most linear growth
\[
|\nabla_x c(x,y)| + |\nabla_y c(x,y)| \le C(1+|x|+|y|).
\]
\item Conditional expectations such as $\mathbb{E}_{p_t}[\nabla_x c(X_t,Y_t)\mid X_t=x]$
exist in $L^2(\mu)$ and define measurable functions of $(x,y)$.
\item 
There exists $L>0$ such that for all couplings $p,q\in\Pi(\mu,\nu)$,
\[
\left\| 
\mathbb{E}_{p}[\nabla_x c(X,Y)\mid X]
-
\mathbb{E}_{q}[\nabla_x c(X,Y)\mid X]
\right\|_{L^2(\mu)}
\;\le\;
L\,W_2(p,q),
\]
and similarly for the conditional expectation with respect to $Y$.
\end{enumerate}
\end{assumption}
\begin{remark}[Stability of the orthogonal complement operator $\SO_p$]
For fixed $p\in\Pi(\mu,\nu)$, the maps
\[
\SO_{p,x}:L^2(p;\mathbb R^n)\to L^2(p;\mathbb R^n),\qquad
\SO_{p,y}:L^2(p;\mathbb R^n)\to L^2(p;\mathbb R^n),
\]
are linear $L^2$-contractions, since they are of the form 
$\SO_{p,x}=I-\PO_{p,X}$ and $\SO_{p,y}=I-\PO_{p,Y}$, where 
$\PO_{p,X},\PO_{p,Y}$ are conditional-expectation operators. Moreover, Assumption~\ref{ass:regularity}(5) implies a stability
of $\SO_p$ with respect to the law $p$ in the directions 
$f=\nabla_x c$ and $f=\nabla_y c$. Indeed, for any
$p,q\in\Pi(\mu,\nu)$,
\[
\|\SO_{p,x}[\nabla_x c]-\SO_{q,x}[\nabla_x c]\|_{L^2(\mu)}
=
\|\mathbb E_q[\nabla_x c\mid X] - \mathbb E_p[\nabla_x c\mid X]\|_{L^2(\mu)}
\le
L\,W_2(p,q),
\]
and similarly for $S_{p,y}[\nabla_y c]$. 
\end{remark}
\begin{remark}[On the tangent space $\textrm{Tan}_{p_t,2}\Pi(\mu,\nu)$]\label{remark:kernel_space}
Note that
\[
\textrm{Tan}_{p_t,2}\Pi(\mu,\nu)
=
\ker(\PO_{p_t,X}[v_1])\times\ker( \PO_{p_t,Y}[v_2]),
\]
where by definition
\[
\PO_{p,X}[v_1]=\mathbb{E}_p[v_1(X,Y)\mid X],
\qquad
\PO_{p,Y}[v_2]=\mathbb{E}_p[v_2(X,Y)\mid Y].
\]
Since conditional expectation is a bounded linear operator, their kernels are closed linear
subspaces of $L^2(p;\mathbb{R}^n)$.
Thus $\textrm{Tan}_{p_t,2}\Pi(\mu,\nu)$ is a closed linear subspace of
$L^2(p;\mathbb{R}^{2n})$.
\end{remark}

\subsection{Proofs}

\begin{proof}[Proof of Proposition~\ref{prop:marg-pres} (Marginal preservation)]
By definition of the tangent space $\textrm{Tan}_{p_t,2} \Pi(\mu,\nu)$ we have
\begin{eqnarray}
\mathbb{E}_{p_t}[v_1(X_t,Y_t)\mid X_t]=0,
\qquad
\mathbb{E}_{p_t}[v_2(X_t,Y_t)\mid Y_t]=0. \label{eq:constraint_proof}
\end{eqnarray}
Let $g\in C_c^\infty(\mathbb{R}^n)$ be a smooth, compactly supported test function.
Using Assumption~\ref{ass:regularity}(1)--(3), the chain rule for the 
ODE~\eqref{eq:main-ODE} yields
\[
\frac{d}{dt}\mathbb{E}[g(X_t)]
=
\mathbb{E}\big[ \nabla g(X_t)\cdot v_{1,t}(X_t,Y_t) \big],
\]
where the exchange of derivative and expectation is justified because 
$\nabla g$ is bounded and $v_t\in L^2(p_t)$.
Using the self-adjointness of conditional expectation 
(see property \ref{prop:self}) we obtain
\[
\mathbb{E}\big[ \nabla g(X_t)\cdot v_{1,t}(X_t,Y_t) \big]
=
\mathbb{E}\big[\nabla g(X_t)\cdot 
\mathbb{E}[v_{1,t}(X_t,Y_t)\mid X_t]\big].
\]
By \eqref{eq:constraint_proof} the right-hand side vanishes, hence
$\frac{d}{dt}\mathbb{E}[g(X_t)]=0$ for all $g\in C_c^\infty(\mathbb{R}^n)$.
This implies that the marginal of $X_t$ is constant in time, i.e.\ $X_t\sim\mu$.
The argument for $Y_t$ is identical.
\end{proof}

\noindent In particular, along the OCD dynamics the marginals remain fixed:
if $(X_0,Y_0)\sim \pi\in\Pi(\mu,\nu)$, then $X_t\sim\mu$ and $Y_t\sim\nu$ for all $t\ge0$.

\begin{proof}[Proof of Proposition~\ref{prop:pos} (Descent in Cost)]
By definition of the cost functional,
\[
\mathcal C(t)
:= \mathbb{E}[c(X_t,Y_t)] 
= \int_{\mathbb{R}^{2n}} c(x,y)\, p_t(dx,dy).
\]
Using that $c\in C^1$ with at most linear growth of its gradients and that $v_t\in L^2(p_t)$
(Assumption~\ref{ass:regularity}), we may differentiate under the expectation to obtain
\begin{eqnarray}
\frac{d}{dt}\mathcal C(t)
=
\mathbb{E}\!\left[
\nabla_x c(X_t,Y_t)\cdot v_{1,t}(X_t,Y_t)
+
\nabla_y c(X_t,Y_t)\cdot v_{2,t}(X_t,Y_t)
\right].
\label{eq:chain-rule-cost}
\end{eqnarray}
For OCD, the velocity field is
\[
v_{1,t}=\nabla_y c - \mathbb{E}[\nabla_y c \mid X_t],
\qquad
v_{2,t}=\nabla_x c - \mathbb{E}[\nabla_x c \mid Y_t].
\]
Substituting these into \eqref{eq:chain-rule-cost} and rearranging gives
\[
\frac{d}{dt}\mathcal C(t)
=
-\mathbb{E}\!\left[
(\nabla_x c - \mathbb{E}[\nabla_x c\mid Y_t])\cdot \nabla_x c
+
(\nabla_y c - \mathbb{E}[\nabla_y c\mid X_t])\cdot \nabla_y c
\right].
\]
Using the self-adjointness of conditional expectation in $L^2(p_t)$
(see property~\ref{prop:self}), we obtain
\[
\mathbb{E}\big[\,\nabla_x c \cdot \mathbb{E}[\nabla_x c\mid Y_t]\,\big]
=
\mathbb{E}\big[\,\mathbb{E}[\nabla_x c\mid Y_t] \cdot \mathbb{E}[\nabla_x c\mid Y_t]\,\big],
\]
and similarly for the $y$-component.
Hence
\begin{eqnarray}
\frac{d}{dt}\mathcal C(t)
=
-
\left(
\|\nabla_x c - \mathbb{E}[\nabla_x c\mid Y_t]\|_{L^2(p_t)}^2
+
\|\nabla_y c - \mathbb{E}[\nabla_y c\mid X_t]\|_{L^2(p_t)}^2
\right). \label{eq:cost-decay}
\end{eqnarray}
Since conditional expectation is an $L^2$-contraction 
(see property~\ref{prop:cont}),
each term in \eqref{eq:cost-decay} is nonnegative, so the right-hand side is
nonpositive. Therefore
\[
\frac{d}{dt}\mathcal C(t)\le 0,
\]
with equality if and only if the two OCD velocity components vanish
$p_t$-almost surely. 
\end{proof}

\begin{proof}[Proof of Proposition \ref{prop:sub-opt} (Instability of Sub-Optimal Couplings)]{
First note that marginals of $p_T^\epsilon$ are $\mu$ and $\nu$ (due to Proposition \ref{prop:marg-pres}). Next, following the definition of $\pi^{\textrm{opt}}_{\mu,\nu}$ (see Eq.~\eqref{eq:pi_opt}), we have  
\begin{eqnarray}
\int_{\mathbb{R}^{2n}}c(x,y)p_T^\epsilon (dx,dy) < \int_{\mathbb{R}^{2n}}c(x,y)p_T(dx,dy) \ .
\end{eqnarray}
However since 
\begin{eqnarray}
\frac{d}{d t}\mathbb{E}[c(X,Y)]  \leq 0 
\end{eqnarray}
due to Proposition \ref{prop:pos},  we have
\begin{eqnarray}
\int_{\mathbb{R}^{2n}}c(x,y)p_t (dx,dy) &\leq & \int_{\mathbb{R}^{2n}}c(x,y)p_T^\epsilon (dx,dy)< \int_{\mathbb{R}^{2n}}c(x,y)p_T(dx,dy) 
\end{eqnarray}
for all $t>T$. Thus given $t>T$,
$\exists A\in\mathcal{B}(\mathbb{R}^{2n})$ for which $p_{t}(A) \neq p_T(A)$.
}
\end{proof}

\begin{proof}[Proof of Theorem~\ref{theorem:vlasov} (McKean--Vlasov system)]
Let $g\in C_c^\infty(\mathbb{R}^{2n})$ be a smooth, compactly supported test
function. Since $c\in C^1$ with at most linear growth of its gradients, and $v_t\in L^2(p_t)$
(Assumption~\ref{ass:regularity}), the chain rule for the ODE
\eqref{eq:main-ODE} and dominated convergence yield
\begin{equation}
\label{eq:vlasov-weak-derivative}
\frac{d}{dt}\mathbb{E}[g(X_t,Y_t)]
=
\mathbb{E}\!\left[
\nabla_x g(X_t,Y_t)\cdot v_{1,t}(X_t,Y_t)
+
\nabla_y g(X_t,Y_t)\cdot v_{2,t}(X_t,Y_t)
\right].
\tag{3.22}
\end{equation}
Since $\mathbb{E}[g(X_t,Y_t)] = \int g(x,y)\,\rho_t(x,y)\,dx\,dy$, the left-hand
side of \eqref{eq:vlasov-weak-derivative} equals
\[
\frac{d}{dt}\int_{\mathbb{R}^{2n}} g(x,y)\,\rho_t(x,y)\,dx\,dy.
\]
\noindent We now rewrite the right-hand side of \eqref{eq:vlasov-weak-derivative}.
Recall that
\[
v_{1,t}(x,y)
=
\nabla_x c(x,y)
-
\mathbb{E}_{p_t}[\nabla_x c(X_t,Y_t)\mid X_t=x]
=
\SO_{p_t,x}[\nabla_x c](x,y),
\]
and analogously
\(
v_{2,t} = \SO_{p_t,y}[\nabla_y c].
\)
Since $g$ has compact support and $v_t\rho_t\in L^1(\mathbb{R}^{2n})$
(Assumption~\ref{ass:regularity}(2)), we may apply integration by parts
{in the sense of distributions} to obtain
\[
\int_{\mathbb{R}^{2n}}
\nabla_x g \cdot v_{1,t}\,\rho_t \ dxdy
=
-\int_{\mathbb{R}^{2n}}
g(x,y)\,\nabla_x\!\cdot\!\big(\rho_t(x,y)\,v_{1,t}(x,y)\big)\,dx\,dy,
\]
and likewise for the $y$--term
\[
\int_{\mathbb{R}^{2n}}
\nabla_y g \cdot v_{2,t}\,\rho_t\ dxdy
=
-\int_{\mathbb{R}^{2n}}
g(x,y)\,\nabla_y\!\cdot\!\big(\rho_t(x,y)\,v_{2,t}(x,y)\big)\,dx\,dy.
\]
\noindent Substituting these identities into
\eqref{eq:vlasov-weak-derivative}, we obtain
\[
\frac{d}{dt}\int_{\mathbb{R}^{2n}} g\,\rho_t \ dxdy
=
-\int_{\mathbb{R}^{2n}}
g\,\nabla_x\!\cdot(\rho_t v_{1,t})\ dxdy
-
\int_{\mathbb{R}^{2n}}
g\,\nabla_y\!\cdot(\rho_t v_{2,t}) \ dxdy.
\]
Since $g$ is an arbitrary test function in $C_c^\infty(\mathbb{R}^{2n})$, we have the weak formulation
\[
\partial_t\rho_t
+
\nabla_x\!\cdot(\rho_t v_{1,t})
+
\nabla_y\!\cdot(\rho_t v_{2,t})
= 0,
\]
which is the Vlasov equation \eqref{eq:vlasov}. All divergences above
are understood in the distributional sense, i.e. no pointwise derivatives of $\rho_t$
are required.
\end{proof}

\begin{lemma} [Preparatory for Proposition \ref{prop:spd}]
\label{lemma:pos-def}
Consider $A,B \in \mathcal{H}$ and let 
\begin{eqnarray}
R&\defeq&\mathbb{E}[A\otimes A-A\otimes \mathbb{E}[A|B]] \ .
\end{eqnarray}
Therefore 
\begin{eqnarray}
R&\in& S^n_{+} \ .
\end{eqnarray}
\begin{proof}{
First note that $R$ is symmetric due to the self-adjoint property of the conditional expectation (see property \ref{prop:self}). Choose an arbitrary $v\in\mathbb{R}^n$ and let  $L\defeq\sum_{i=1}^n v_iA_i$.  We have
\begin{eqnarray}
\sum_{i,j=1}^nv_i\mathbb{E}[A_i{\mathbb{E}[A_j|B]}]v_j=\mathbb{E}[L{\mathbb{E}[L|B]}]&\leq& \frac{1}{2}\left(\mathbb{E}[\lVert L \rVert_2^2]+\mathbb{E}\left[\lVert{\mathbb{E} [L|B]}\rVert_2^2\right]\right) \ \ \ \ \textrm{(Cauchy-Schwartz)} \nonumber \\
&\leq&\mathbb{E}[\lVert L \rVert_2^2]  \ \ \ \ \ \ \ \ \ \ \ \ \ \ \ \ \ \ \ \ \ \ \ \ \  \textrm{(contraction property \ref{prop:cont})} \nonumber \\
&=&\sum_{i,j=1}^n v_i\mathbb{E}[A_iA_j]v_j  \ . \nonumber
\end{eqnarray}
Therefore $\sum_{i,j}^n v_iR_{ij}v_j \ge0$, and thus $R\in S^n_{+}$. 
}
\end{proof}
\end{lemma}

\begin{proof}[Proof of Proposition~\ref{prop:spd} (Symmetric positive semidefinite correlation)]
Recall that $J_t = \mathbb{E}[X_t\otimes Y_t]$.
Differentiating in time and using the $L^2$-OCD dynamics
\[
\dot X_t = Y_t - \mathbb{E}[Y_t\mid X_t],\qquad
\dot Y_t = X_t - \mathbb{E}[X_t\mid Y_t],
\]
together with Assumption~\ref{ass:regularity} and the chain rule, we obtain
\[
\frac{\partial}{\partial t} J_t
=
\mathbb{E}\big[\dot X_t\otimes Y_t\big]
+
\mathbb{E}\big[X_t\otimes \dot Y_t\big].
\]
Substituting the expressions for $\dot X_t$ and $\dot Y_t$ yields
\begin{align}
\frac{\partial}{\partial t}J_t
&=
\mathbb{E}\big[(Y_t - \mathbb{E}[Y_t\mid X_t])\otimes Y_t\big]
+
\mathbb{E}\big[X_t\otimes (X_t - \mathbb{E}[X_t\mid Y_t])\big]
\nonumber\\
&=
\mathbb{E}\big[Y_t\otimes Y_t - Y_t\otimes \mathbb{E}[Y_t\mid X_t]\big]
+
\mathbb{E}\big[X_t\otimes X_t - X_t\otimes \mathbb{E}[X_t\mid Y_t]\big]
\nonumber\\
&=: R_t^J.
\label{eq:corr-ev}
\end{align}
We claim that $R_t^J\in S_+^n$ for all $t\ge 0$.
Indeed, the first term in \eqref{eq:corr-ev} has the form
\[
\mathbb{E}\big[A\otimes A - A\otimes \mathbb{E}[A\mid B]\big]
\quad\text{with } A=Y_t,\ B=X_t,
\]
and is therefore symmetric positive semidefinite by Lemma~\ref{lemma:pos-def}.
The same lemma applies to the second term in \eqref{eq:corr-ev}, with
$A=X_t$ and $B=Y_t$, so both contributions lie in $S_+^n$ and hence
$R_t^J\in S_+^n$. The
matrix $J_t$ can be written as
\[
J_t = J_0 + \int_0^t R_s^J\,ds.
\]
By assumption $J_0\in S_+^n$,
and the integral of symmetric positive semidefinite matrices is again
symmetric positive semidefinite. Therefore $J_t\in S_+^n$ for all $t\ge0$.
\end{proof}

\begin{lemma}[Orthogonal projections onto marginal–preserving subspaces]
\label{lemma:opt-marg-pres}
Let $(X,Y)\sim\pi\in\Pi(\mu,\nu)$ and consider the Hilbert space 
$L^2(\pi;\mathbb{R}^n)$ with inner product 
$\langle f,g\rangle = \mathbb{E}_\pi[f\cdot g]$.
Define the closed linear subspaces
\[
\mathcal{H}_X^\perp 
:= \{\,\phi\in L^2(\pi;\mathbb{R}^n): 
\mathbb{E}_\pi[\phi(X,Y)\mid X]=0 \,\},
\]
\[
\mathcal{H}_Y^\perp 
:= \{\,\psi\in L^2(\pi;\mathbb{R}^n): 
\mathbb{E}_\pi[\psi(X,Y)\mid Y]=0 \,\}.
\]
Both are closed in $L^2(\pi)$ since conditional expectation is a bounded 
linear operator. Consider the minimization problems
\[
\tilde\phi := 
\argmin_{\phi\in\mathcal{H}_X^\perp}
\mathbb{E}_\pi\!\left[\|\,Y-\phi(X,Y)\,\|_2^2\right],
\]
\[
\tilde\psi := 
\argmin_{\psi\in\mathcal{H}_Y^\perp}
\mathbb{E}_\pi\!\left[\|\,X-\psi(X,Y)\,\|_2^2\right].
\]
Then the minimizers exist, are unique, and satisfy
\[
\tilde\phi(X,Y) = Y - \mathbb{E}_\pi[Y\mid X],
\qquad
\tilde\psi(X,Y) = X - \mathbb{E}_\pi[X\mid Y].
\]
\end{lemma}

\begin{proof}
The functional $\phi\mapsto \mathbb{E}\|Y-\phi\|^2$ is strictly convex, continuous, and 
coercive on $L^2(\pi)$, and $\mathcal{H}_X^\perp$ is a closed linear subspace. 
Hence, there exists a unique minimizer, 
namely the orthogonal projection of $Y$ onto $\mathcal{H}_X^\perp$. Let $\PO$ denote this projection. Since $\mathcal{H}_X^\perp$ is the kernel of 
the bounded operator $\phi\mapsto \mathbb{E}[\phi\mid X]$, we have 
\[
\PO_{p,X} Y = Y - \mathbb{E}_p[Y\mid X].
\]
The same argument applies for $\psi$, yielding 
$\tilde\psi = X - \mathbb{E}[X\mid Y]$.
\end{proof}

\noindent In the next proof we show that the OCD \eqref{eq:main-ODE} gives rise to sharpest decay in quadratic cost among processes evolving  in $\textrm{Tan}_{p_t,2} \Pi(\mu,\nu)$.

\begin{proof}[Proof of Theorem~\ref{theorem:sharp-descent} (Sharp descent and variational formulation)]
For any admissible velocity field 
$v=(v_1,v_2)\in {\ker(\PO_{p_t,X})\times\ker( \PO_{p_t,Y})} $, 
the instantaneous rate of change of the quadratic cost is
\[
\frac{d}{dt}\mathcal{C}(p_t)
=
\mathbb{E}_{p_t}[Y\cdot v_1 + X\cdot v_2].
\]
A direct computation shows that
\[
\mathcal L_{p_t}(v)
=\frac{1}{2}\left(
\mathbb E_{p_t}\big[\|v_1 - Y\|^2 + \|v_2 - X\|^2\big]
-
\mathbb E_{p_t}[\|Y\|^2 + \|X\|^2]\right),
\]
and the second term does not depend on $v$. Therefore minimizing
$\mathcal L_{p_t}(v)$ over the constraint ${\ker(\PO_{p_t,X})\times\ker( \PO_{p_t,Y})}$
is equivalent to minimizing
\[
v\mapsto
\mathbb E_{p_t}\big[\|v_1 - Y\|^2 + \|v_2 - X\|^2\big]
\quad\text{over } {\ker(\PO_{p_t,X})\times\ker( \PO_{p_t,Y})}.
\]
Since ${\ker(\PO_{p_t,X})\times\ker( \PO_{p_t,Y})}$ is a closed linear subspace of $L^2(p_t)$ (see Remark \ref{remark:kernel_space}), and 
$\mathcal{L}_{p_t}$ is strictly convex, continuous, and coercive, there exists a unique minimizer.
By Lemma~\ref{lemma:opt-marg-pres}, the unique minimizer is
\[
v_{1,t}^{\textrm{opt}} = Y - \mathbb{E}[Y\mid X],
\qquad
v_{2,t}^{\textrm{opt}} = X - \mathbb{E}[X\mid Y],
\]
which is precisely the OCD velocity field.
\end{proof}

\begin{remark}
Note there is a similarity between the sharp decay property of OCD and Martingale series approximation employed in \cite{ruschendorf1985wasserstein} (see Lemma 9 and Theorem 10 in \cite{ruschendorf1985wasserstein} ). 
\end{remark}

\section{Characterization in Gaussian Setting}
\label{sec:special_examples}
We analyze the OCD dynamics for quadratic cost with Gaussian marginals and show when it coincides with OT. Let $\mu$ and $\nu$ be centered Gaussian measures with covariance matrices $\Sigma_\mu$ and $\Sigma_\nu$. The dynamics of correlation  matrix $J_t=\mathbb{E}[X_t\otimes Y_t]$, following $L^2$-OCD \eqref{eq:main-ODE}, reads
\begin{eqnarray}
\label{eq:riccati}
\partial_t J_t&=&\Sigma_\mu+\Sigma_\nu-J_t^T\Sigma_\mu^{-1}J_t-J_t\Sigma_\nu^{-1}J_t^T \ ,
\end{eqnarray}
due to the identities 
\begin{eqnarray}
\mathbb{E}_{p_t}[Y_t|X_t]=J_t^T\Sigma_\mu^{-1}X_t \ \ \ \textrm{and} \ \ \ \mathbb{E}_{p_t}[X_t|Y_t]=J_t\Sigma_\nu^{-1}Y_t \ ,
\end{eqnarray}
which results from $p_t$ being Gaussian, see \cite{conforti2023projected}. 
The matrix-Riccati equation \eqref{eq:riccati} admits many stationary solutions, since any bijection between $\mu$ and $\nu$ is a stationary solution of \eqref{eq:main-ODE}, as discussed in \cite{conforti2023projected}. However, as we justified in Proposition \ref{prop:pos}, starting with independent samples, i.e. $\mathbb{E}[X_0Y_0]=0$, $J_t$ remains symmetric positive semi-definite and its trace  increases monotonically. This restricts the stationary state and results in the OT solution if $\Sigma_\mu$ and $\Sigma_\nu$ commute. 
\begin{proposition}[OCD recovers OT for commuting Gaussian covariances]
\label{prop:OCD-Gaussian}
Let 
\[
\mu = \mathcal N(0,\Sigma_\mu), 
\qquad 
\nu = \mathcal N(0,\Sigma_\nu),
\]
with $\Sigma_\mu,\Sigma_\nu \succ 0$, and let $(X_t,Y_t)$ evolve according to the $L^2$-OCD dynamics. If the covariances commute, i.e.
\[
[\Sigma_\mu,\Sigma_\nu]=0,
\]
then $L^2$-OCD converges to the {optimal transport correlation}, namely
\[
{
J_t \;\longrightarrow\; J_\infty
= 
\Sigma_\mu^{1/2}
\big(\Sigma_\mu^{1/2}\Sigma_\nu \Sigma_\mu^{1/2}\big)^{1/2}
\Sigma_\mu^{1/2}
= J_{\mathrm{OT}}.
}
\]
\end{proposition}

\begin{proof}
When $[\Sigma_\mu,\Sigma_\nu]=0$, the two matrices are simultaneously diagonalizable:
\[
\Sigma_\mu = Q\,\mathrm{diag}(\sigma_{\mu,i})\,Q^\top,
\qquad
\Sigma_\nu = Q\,\mathrm{diag}(\sigma_{\nu,i})\,Q^\top,
\]
for some orthogonal $Q$.  
Under this change of coordinates, \eqref{eq:riccati} decouples component-wise
\[
\dot j_i 
= \sigma_{\mu,i} + \sigma_{\nu,i}
  - (\sigma_{\mu,i}^{-1}+\sigma_{\nu,i}^{-1})\,j_i^2,
\qquad j_i(0)=0.
\]
By direct computation we find that the equilibrium (given $j_i(0)=0$) is
\[
j_i^\star 
= \sqrt{\sigma_{\mu,i}\sigma_{\nu,i}}\ .
\]
Hence $J_t \to J_\infty = Q\,\mathrm{diag}(\sqrt{\sigma_{\mu,i}\sigma_{\nu,i}})\,Q^\top.$ On the other hand, the Gaussian OT cross--correlation is 
\[
J_{\mathrm{OT}}
=\Sigma_\mu^{1/2}
(\Sigma_\mu^{1/2}\Sigma_\nu \Sigma_\mu^{1/2})^{1/2}
\Sigma_\mu^{1/2},
\]
see e.g. \cite{gelbrich1990formula}. When $\Sigma_\mu$ and $\Sigma_\nu$ commute, we get
\[
J_{\mathrm{OT}}
= Q\,\mathrm{diag}(\sqrt{\sigma_{\mu,i}\sigma_{\nu,i}})\,Q^\top
= J_\infty.
\]
\end{proof}
\begin{remark}
When the covariances commute (i.e., the Gaussian marginals share eigenvectors), OCD reduces to $n$ independent scalar Riccati flows, each converging to the correct cross--variance $\sqrt{\sigma_{\mu,i}\sigma_{\nu,i}}$. 
If the covariances do not commute, the symmetry of $J_t$ is preserved, preventing convergence to the generally non-symmetric OT solution.
\end{remark}
Finally we address the speed of convergence to the OT solution in the case of commutting covariance matrices. 
\begin{theorem}[Exponential convergence of OCD to OT in the commuting Gaussian case]\label{thm:expon-gauss}
Let $\mu=\mathcal N(0,\Sigma_\mu)$ and $\nu=\mathcal N(0,\Sigma_\nu)$ with
$\Sigma_\mu,\Sigma_\nu\succ0$ and $[\Sigma_\mu,\Sigma_\nu]=0$. 
Consider the $L^2$-OCD dynamics initialized with independent samples,
and let $J_t=\mathbb E[X_tY_t^\top]$ denote the cross--correlation.
Let
\[
J_{\textrm{OT}} \;=\; \Sigma_\mu^{1/2}\big(\Sigma_\mu^{1/2}\Sigma_\nu\Sigma_\mu^{1/2}\big)^{1/2}\Sigma_\mu^{1/2}
\]
be the Gaussian OT cross--correlation. 
Then $J_t\to J_{\textrm{OT}}$ exponentially fast: there exist constants 
$C>0$ and $\kappa>0$ such that
\[
\|J_t-J_{\textrm{OT}}\|_2 \;\le\; C\,e^{-2\kappa t}\qquad\text{for all }t\ge0,
\]
where $\|\cdot\|_2$ denotes the spectral norm. Expressing $\Sigma_\mu=Q\diag(\sigma_{\mu,i})Q^\top$ and
$\Sigma_\nu=Q\diag(\sigma_{\nu,i})Q^\top$ in a common orthonormal basis,
\[
\kappa \;=\; \min_{1\le i\le n}\Big(\sqrt{\tfrac{\sigma_{\nu,i}}{\sigma_{\mu,i}}}
+\sqrt{\tfrac{\sigma_{\mu,i}}{\sigma_{\nu,i}}}\Big)\;\ge\;2.
\]
\end{theorem}

\begin{proof}
Since $[\Sigma_\mu,\Sigma_\nu]=0$, there exists orthogonal $Q$ with
$\Sigma_\mu=Q\diag(\sigma_{\mu,i})Q^\top$ and
$\Sigma_\nu=Q\diag(\sigma_{\nu,i})Q^\top$.
The Riccati ODE
\[
\dot J_t=\Sigma_\mu+\Sigma_\nu - J_t^\top\Sigma_\mu^{-1}J_t - J_t\,\Sigma_\nu^{-1}J_t^\top
\]
preserves symmetry, so writing $J_t=Q\diag(j_i(t))Q^\top$ gives $n$ decoupled scalar ODEs
\[
\dot j_i(t)=\sigma_{\mu,i}+\sigma_{\nu,i} - (\sigma_{\mu,i}^{-1}+\sigma_{\nu,i}^{-1})\,j_i(t)^2
=\alpha_i\big((j_i^\star)^2 - j_i(t)^2\big),
\]
where $\alpha_i:=\sigma_{\mu,i}^{-1}+\sigma_{\nu,i}^{-1}$ and $j_i^\star:=\sqrt{\sigma_{\mu,i}\sigma_{\nu,i}}$.
The explicit solution is
\[
j_i(t)=j_i^\star\,\tanh\!\big(\alpha_i j_i^\star t \big).
\]
Hence
\[
|j_i(t)-j_i^\star|
= j_i^\star\Big|\,\tanh(\alpha_i j_i^\star t )-1\Big|
= \frac{2j_i^\star\,e^{-2(\alpha_i j_i^\star t )}}{1+e^{-2(\alpha_i j_i^\star t )}}
\;\le\; 2j_i^\star\,e^{-2\rho_i t},
\]
with $\rho_i:=\alpha_i j_i^\star=\sqrt{\sigma_{\nu,i}/\sigma_{\mu,i}}+\sqrt{\sigma_{\mu,i}/\sigma_{\nu,i}}$.
Therefore, for the spectral norm we have
\[
\|J_t-J_{\textrm{OT}}\|_2=\big\|Q\,\diag(j_i(t)-j_i^\star)\,Q^\top\big\|_2
\le \max_i |j_i(t)-j_i^\star|
\le 2\max_i j_i^\star\,e^{-2(\min_i\rho_i)t}.
\]
Setting $C:=2\max_i j_i^\star$ and $\kappa:=\min_i\rho_i$ gives the claim. 
Finally, $\rho_i\ge 2$ (by arithmetic-geometric mean inequality), hence $\kappa\ge 2$.
\end{proof}

\section{Numerical Algorithm through Opinion Dynamics}  
\label{sec:numerics}
\noindent  Opinion dynamics at its core tries to answer how different agents shape their opinions in contact with each-other \cite{kolarijani2020macroscopic}. While seemingly unrelated to the Monge-Kantorovich problem, in this section, we leverage analogy between opinion dynamics and OCD to provide a novel, yet simple, numerical algorithm for the latter. The main idea is to first, split the data points into a certain number of clusters. Next, the conditional expectation is estimated in each cluster using an appropriate hypothesis class.
\subsection{Analogy with Opinion Dynamics}
\noindent Consider a population of agents indexed by $i\in\{1,...,N\}$. Suppose each person $i$ has certain two-dimensional position $P_i(t)\in\mathbb{R}^2$ as their opinion in time $t$. In the standard form of  the bounded interval Hegselmann-Krause consensus dynamics, e.g., see \cite{hegselmann2006truth,motsch2014heterophilious}, these positions relax towards a global consensus via 
\begin{eqnarray}
\dot{P}_i&=&\sum_{i\neq j}a_{ij}(P_i(t)-P_j(t)) \ ,
\end{eqnarray}
where $a_{ij}\ge 0$ captures the influence of other agents positions. We can decompose it as $a_{ij}=\phi(r_{ij})\alpha_{ij}$ where $r_{ij}=\lVert P_i-P_j\rVert_2^2$ is the Euclidean distance between the positions of two persons $i$ and $j$, and $\alpha$ a constant matrix. Based on $\alpha$, we can intuitively consider two idealized regimes:
\begin{enumerate}
\item The positions are completely independent from each-other and $\alpha$ becomes the identity matrix.
\item  The positions are perfectly tied to each-other and $\alpha$ becomes the exchange matrix $J$ (with ones on its anti-diagonal and zeros elsewhere).
\end{enumerate} 
The first one degenerates the system into independent positions. More interestingly, let us focus on the latter and denote $P=[P_X,P_Y]^T$, therefore the dynamics takes the form
\begin{eqnarray}
\dot{P}_{X,i}=\sum_{i\neq j}\phi(r_{ij})(P_{Y,i}(t)-P_{Y,j}(t)) \ \ \textrm{and} \ \ 
\dot{P}_{Y,i}=\sum_{i\neq j}\phi(r_{ij})(P_{X,i}(t)-P_{X,j}(t)) . \nonumber
\end{eqnarray}
If we assume $\phi(\ )$ is simply a step function i.e. it is one if $r\le \epsilon$ and zero otherwise (for $\epsilon>0$), we can further simplify the dynamics to 
\begin{eqnarray}
\dot{P}_{X,i}=\frac{1}{N_{\bar{I}^{X_i,\epsilon}}}\sum_{j\in\bar{I}^{X_i,\epsilon}} (P_{Y,i}(t)-P_{Y,j}(t)) \ \ \textrm{and} \ \ 
\dot{P}_{Y,i}=\frac{1}{N_{\bar{I}^{Y_i,\epsilon}}}\sum_{j\in\bar{I}^{Y_i,\epsilon}} (P_{X,i}(t)-P_{X,j}(t)) \ , \nonumber
\end{eqnarray}
where $\bar{I}^{X_i,\epsilon}$ and  $\bar{I}^{Y_i,\epsilon}$ are the set of all agents with $\lVert P_{X,i}-P_X\rVert_2^2\leq \epsilon^2$ and $\lVert P_{Y,i}-P_Y\rVert_2^2\leq \epsilon^2$, respectively, with the corresponding number of elements $N_{\bar{I}^{X_i,\epsilon}}$ and $N_{\bar{I}^{Y_i,\epsilon}}$. In the limit of $N\to\infty$ and $\epsilon\to 0$, it is suggestive to express the above dynamics in the statistical form. Treating $P_{X,i}(t)$ and $P_{Y,i}(t)$ as samples of random variables $X_t\in \mathcal{H}$ and $Y_t\in \mathcal{H}$, respectively, the dynamics takes the form
\begin{eqnarray}
\label{eq:OT-ODE}
\dot{X}_t=Y_t-\mathbb{E}[Y_t|X_t] \ \ \ \textrm{and} \ \ \
\dot{Y}_t=X_t-\mathbb{E}[X_t|Y_t]  \ ,
\end{eqnarray}
which is OCD \eqref{eq:main-ODE} for $L^2$-cost. From technical point of view, the discrete version is justified due to the approximation of conditional expectation by a discrete kernel \cite{nadaraya1964estimating}. Through the lens of opinion dynamics, the OCD describes a dynamics which is local in a sense that each agent is only influenced by agents in its vicinity, yet as discussed before, it is closely linked to the global optimization problem \eqref{eq:pi_opt}. 
\subsection{Non-parametric Monte-Carlo Algorithm} 
\noindent The analogy with opinion dynamics hints a numerical recipe for OCD. Conceptually, the parameter $\epsilon$ gives rise to formation of clusters in the cloud of $(X_t,Y_t)$ points. We consider two numerical schemes:
\begin{enumerate}
\item {\bf OCD-piecewise constant.} In each cluster, $X_t$ and $Y_t$ relax towards their locally averaged values, i.e. the conditional expectation is estimated to be constant in the cluster. 
\item {\bf OCD-piecewise linear.} In each cluster, the conditional expectation is estimated using linear regression, i.e. the joint distribution of each cluster is estimated by a Gaussian.
\end{enumerate}
Starting with $N_p$, i.i.d., samples $\hat{X}_{i,t_0}\sim \mu$ and $\hat{Y}_{i,t_0}\sim \nu$, we update the state of each sample according to discretized form of OCD. It is important to note that here we are focusing on an approximate solution of the optimization problem \eqref{eq:pi_opt} over atom-less distributions $\mu$ and $\nu$. An alternative interpretation, where the distributions are assumed to be sum of diracs (assignment problem) is not pursued here (see \cite{trigila2016data}). Therefore throughout what follows, we assume that there exists an underlying continuous map between the two measure spaces.
Three hyper-parameters control the discretization error: number of samples $N_p$, time step size $\Delta t$, and cluster cut-off $\epsilon$. Once these values are fixed, relying on Euler's scheme (see Runge-Kutta's scheme in Supplementary Information), each sample is updated according to 
\begin{eqnarray}
\hat{X}^{n+1}_i&=&\hat{X}^n_i+\nabla_x c(\hat{X}_i^n,\hat{Y}_i^n)\Delta t-\hat{\mathcal{K}}^n_X\Delta t \\
\textrm{and} \ \ \ 
\hat{Y}^{n+1}_i&=&\hat{Y}^n_i+\nabla_y c(\hat{X}_i^n,\hat{Y}_i^n)\Delta t-\hat{\mathcal{K}}^n_Y \Delta t \ ,
\end{eqnarray}
where superscript $n$ and $n+1$ denote the approximations at $t^n$ and $t^{n}+\Delta t$, respectively.  The functionals $\hat{\mathcal{K}}^n_{X,Y}(.)$ estimate conditional expectations with respect to $X$ and $Y$, respectively, and they are found depending on the following estimators. 
\begin{enumerate}
\item  {\bf OCD-piecewise constant.} The estimator is constant in each cluster (equivalent to using square kernel in the estimation of the conditional expectation). The expressions for $\hat{\mathcal{K}}^n_{X,Y}(.)$ are given in  Supplementary Information. 
\item {\bf OCD-piecewise linear.} The estimator follows linear regression (accounts for using Gaussian kernel in the estimation of the conditional expectation). In other words here we look for $L^2$ projection of $\nabla_{x,y} c$ onto the span of linear functions \cite{klebanov2021linear}. The final expressions are given in Supplementary Information. 
\end{enumerate}
The overview of the algorithm is given in Supplementary Information. 
The convergence of the algorithm, in general, is controlled by $N_p, \Delta t, \epsilon$.  Assuming that the OCD initialized by samples of $\mu,\nu$ is well-posed, it is reasonable to consider fixed-order convergence with respect to $\Delta t$ (e.g. first-order for Euler time integration, see \cite{petzold1983automatic}). Also the error due to finite $N_p$ is expected to behave similar to conventional Monte-Carlo schemes \cite{kroese2013handbook}. However, the error dependency on the cut-off value $\epsilon$ is more delicate. 
Two asymptotic limits of $\epsilon\to 0$ and $\epsilon \to \infty$ can be intuitively understood. The limit $\epsilon\to 0$ accounts for no interaction between agents, i.e. each point forms its own cluster. Therefore the points remain frozen in time. On the other hand, $\epsilon\to \infty$ gives rise to perfect interaction among all agents, forming one giant cluster. In this case, the dynamics converges to the average value among all sample points, for the 1st order approximation,  and towards linear regression, for the 2nd order case. Obviously, the optimal value of $\epsilon$ needs to be tuned in-between, and is expected to depend on $N_p$. The transition between frozen dynamics in $\epsilon\to 0$ and homogenization $\epsilon\to \infty$ suggests a phase transition type behaviour, resulting from $\epsilon$. The transition point can serve as a proxy for the optimal cut-off value, as discussed in Supplementary Information.   \\ \ \\
Once the pairing between $\hat{X}_i$ and $\hat{Y}_i$ reaches its stationary state (according to a stopping criterion), the sampled points $(\hat{X}_i^s,\hat{Y}_i^s)_{i\in\{1,...,N_p\}}$ can be employed to estimate the Monge map $\hat{T}()$, using a suitable hypothesis class. We train a Neural Network (NN) with a few layers and $\tanh(.)$ as the activation function to approximate the map and denote it by $M_{X\rightarrow Y}$. As the loss function, we simply consider the expectation of $L^2$ point-wise error between NN's output $M_{X\rightarrow Y}(\hat X^*)$ and $\hat Y^*$, i.e. $ \mathrm{Loss}=\sum_{i=1}^{N_p} \lVert M_{X\rightarrow Y}(\hat X_i^*)- \hat Y_i^*\rVert_2^2/N_p$. 
\subsection {Computational Complexity} 
\noindent  In the following, we provide complexity estimates of OCD. Note that the discussion is mainly based on heuristic arguments and rely 
on practical scaling observations rather than formal asymptotic analysis. For a given time step (iteration), the main computational task of the proposed OCD algorithm is the evaluation of the conditional expectation, whose estimation relies on forming clusters by finding neighbours of each point in a given distribution. For the latter, though non-exclusively, the Ball Tree method \cite{omohundro1989five} scales efficiently i.e. $\mathcal{O}(n N_p \log N_p)$ with number of particles $N_p$  and space dimension $n$.  Therefore, we expect that the OCD equipped with piecewise constant estimator for the conditional expectation to have the time complexity of $\mathcal{O}(n N_p \log N_p)$, since the number of ODEs is $2n$ and number of operations in order to compute the distance between each two particles is $\mathcal{O}(n)$.
However, since we deployed the direct solver in the computation of the conditional expectation with piecewise linear estimator, the corresponding OCD approach scales poorly with space dimension $n$ leading to the overall time complexity of $\mathcal{O}(n^3N_p \log N_p)$.
\\ \ \\
In terms of memory consumption, since OCD, unlike linear programming, does not require computation of the distance matrix, its memory requirement is
linear with the number of particles (histogram-dimension) $N_p$. In terms of space dimension $n$, OCD becomes memory intensive at high $n$ since the deployed direct linear solver scales poorly in the high dimensional setting.
\\ \ \\
In order to estimate the total computational cost of OCD, including number of iterations required to reach stationary state, we  conducted numerical tests as shown in Fig.~\ref{fig:cost_mem_norm_m0v1_norm_m1v1}, where OCD runs between two normal distributions. The flops and storage requirements confirm our analysis on complexity and memory consumption. However, it is evident that OCD requires significantly more operations compared to the reference Earth Mover's Distance (EMD) method \cite{JMLR:v22:20-451}, if relatively small number of sample points (e.g. 100) is considered. This is mainly due to the fact that while the number of operations scales efficiently, i.e. with $\mathcal{O}(N_p \log N_p)$, the pre-factor which depends on the number of iterations required to reach stationary condition becomes dominant, if $N_p$ is relatively small. For $N_p>64'000$, we observe that OCD becomes more efficient than EMD. On the other hand, as expected, OCD with the piecewise linear approximation of the conditional expectation scales poorly with the dimension, i.e. $\mathcal{O}(n^{3.5})$, while the measured timing of OCD with the piecewise constant estimator of the conditional expectation is slightly larger than expected $\mathcal{O}(n)$. 
We performed this study using the Google's Cloud TPU v6e with 172.9 GB RAM.

\begin{figure}
  	\centering
   \begin{tabular}{cc}
\includegraphics[scale=0.63]{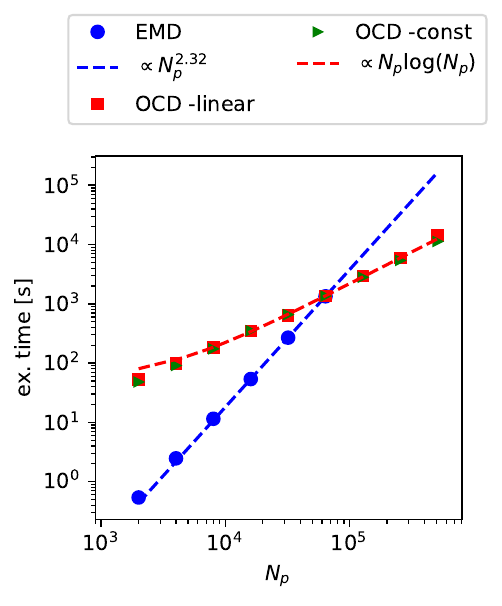}
&
\includegraphics[scale=0.63]{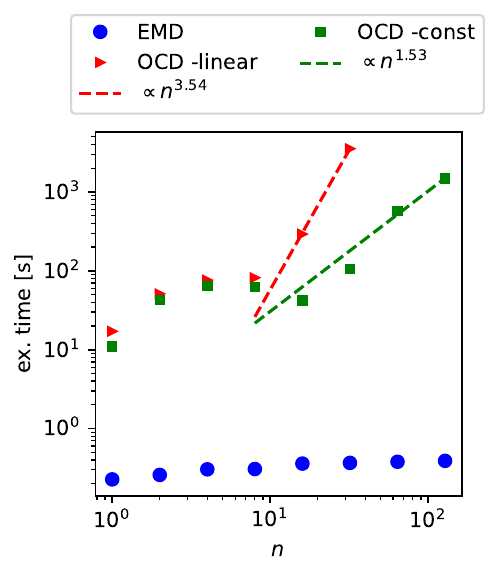}
\\
\includegraphics[scale=0.63]{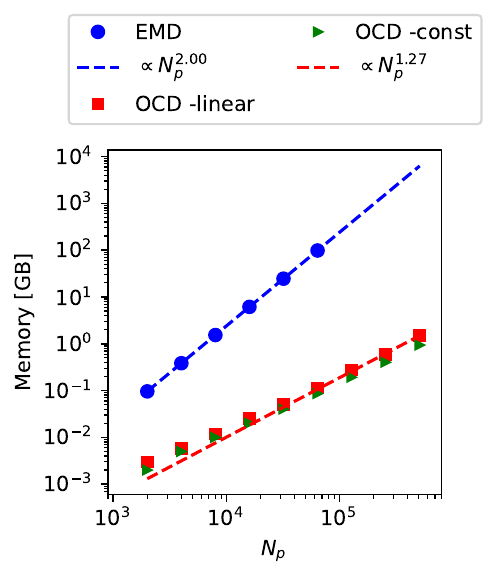}
&
\includegraphics[scale=0.63]{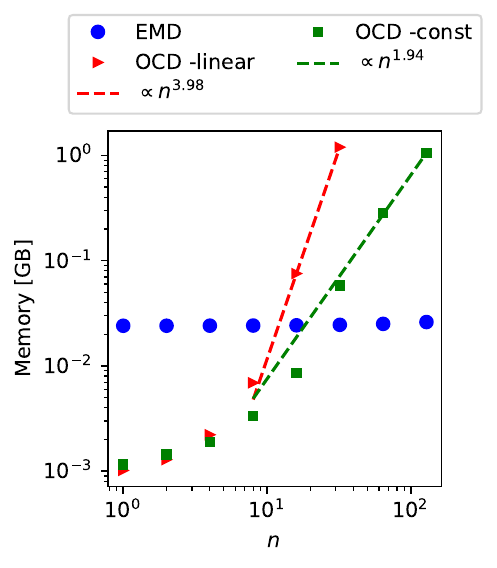}
\\
(a) $n=3$ & (b) $ N_p=1000$
   \end{tabular}
   \caption{Execution time (top) and memory consumption (down) of OCD-piecewise linear algorithm with Runge Kuta 4th order solver with $\Delta t=0.1$ against Earth Mover's Distance (EMD) method \cite{JMLR:v22:20-451} in learning the map between samples of $\mu=\mathcal N( 0, I)$ and $\nu=\mathcal N( 1, I)$ in (a) $n=3$ for a range of $N_p$ and (b) $n=1,...,128$ with $ N_p=1000$. Here we scale $\epsilon$ with a rule of thumb $\eps = 3 n {N_p}^{-1/4}/4$.}
   \label{fig:cost_mem_norm_m0v1_norm_m1v1}
\end{figure}

\section{Representative Results for $L^2$-Monge-Kantorovich} 
\label{sec:results}
\noindent In this section, we illustrate the accuracy and cost of the proposed OCD and highlight its capability in estimating the optimal map, along with Wasserstein distance, between two marginals in a few examples. Unless mentioned otherwise, by OCD we are referring to OCD-piecewise linear algorithm equipped with Runge-Kutta 4th order ODE solver.
We refer the reader to Supplementary Information for details on the choice of hyper-parameters such as $\epsilon$ and $\Delta t$, and comparison with benchmarks. All the performance studies are done on a single core and thread of Intel Core i7-8550U CPU with 1.80 GHz frequency.

\subsection{Clustering and Prototypical Examples}

\begin{figure}
  	\centering
\includegraphics[scale=0.03]{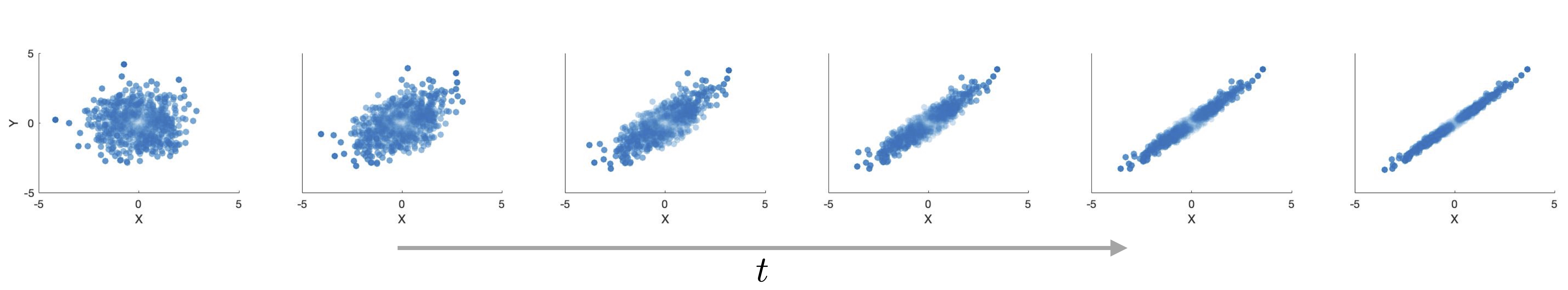}
\caption{Evolution of point particles under the coupling created by OCD. }
   \label{fig:clustering}
\end{figure}
\begin{figure}
  	\centering
\includegraphics[scale=0.06]{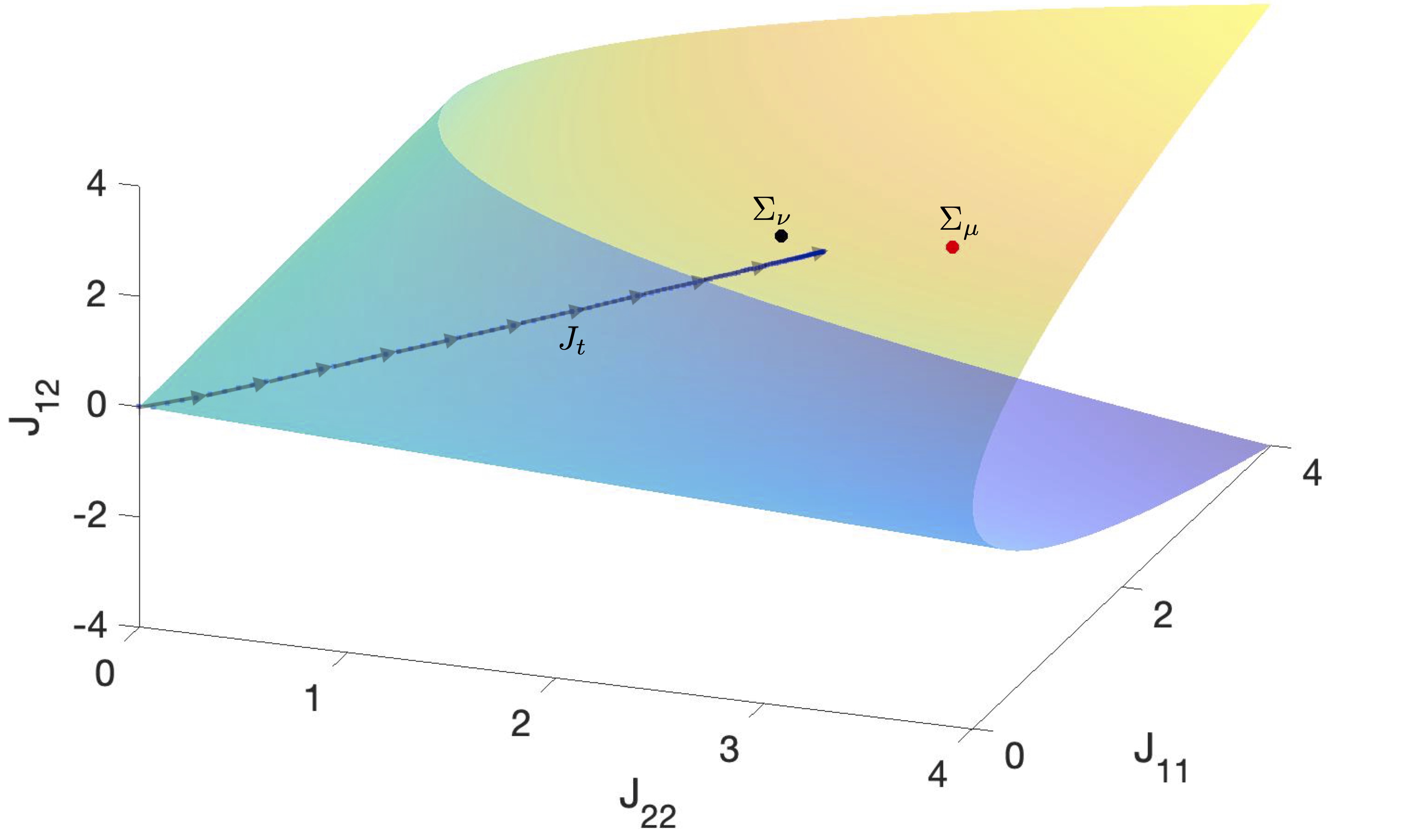}
\caption{Evolution of cross-correlation matrix, $J_t$, in the cone of symmetric positive semi-definite matrices. }
   \label{fig:psd_cone}
\end{figure}
\begin{figure}
  	\centering
\includegraphics[scale=0.025]{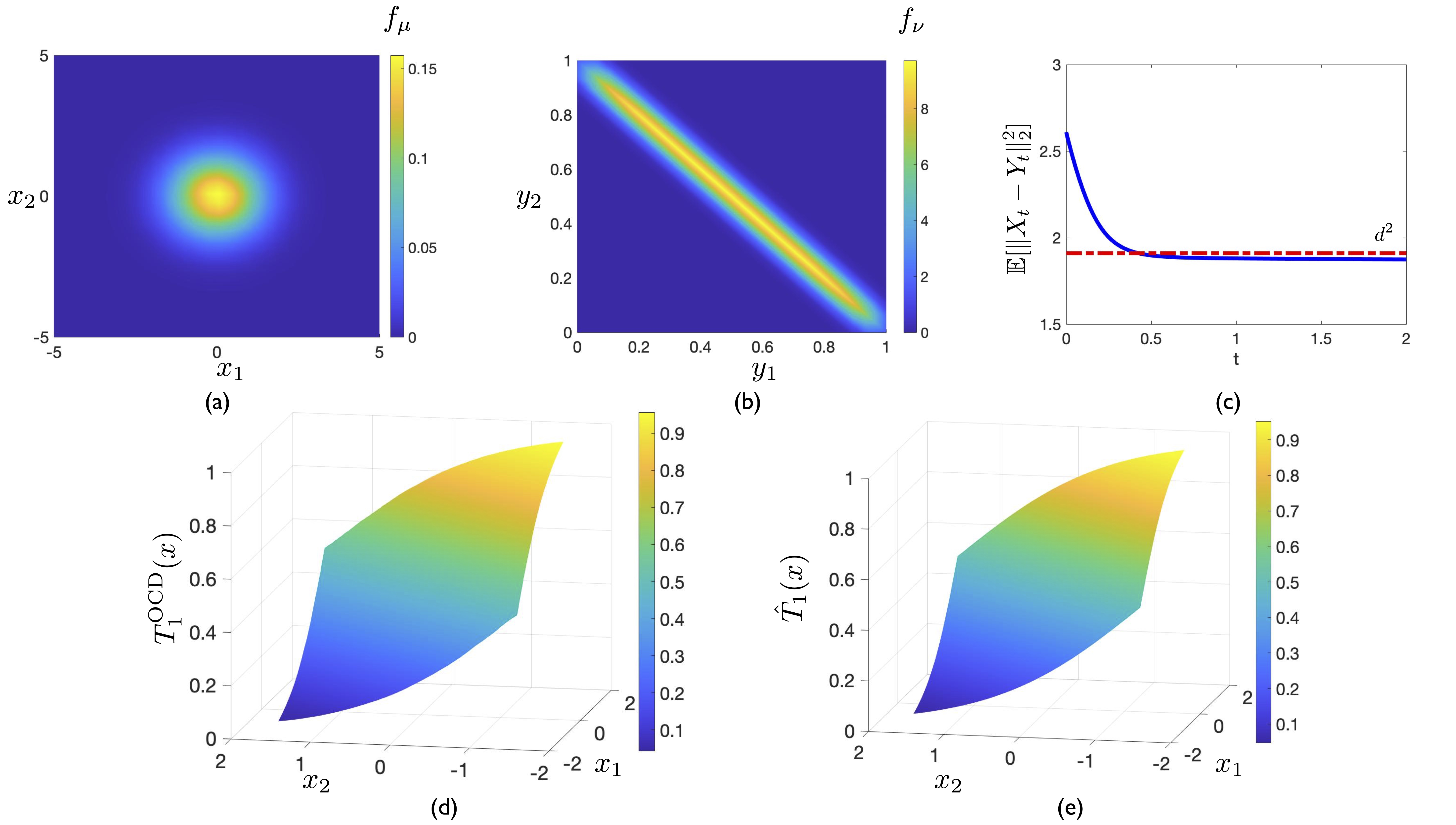}
\caption{Recovery of nonlinear Monge map between (a) normal distribution and (b) its push-forward by softmax map $T_i=\nabla_{x_i} \log \left(\exp(x_1)+\exp(x_2)\right)$; (c) Starting from independent samples, the pair $(X_t,Y_t)$ is evolved by OCD and $\mathbb{E}\lVert X_t-Y_t \rVert_2^2$ converges towards the Wasserstein distance $d^2$; (d) First component of the map generated by OCD; (e) First component of the exact Monge map.}
   \label{fig:softmax}
\end{figure}
\noindent First, we present and discuss three toy scenarios.
\begin{enumerate}
\item {\bf Particle Clustering.} Starting  from independent samples of two one-dimensional standard normal distributions, we track the particles evolved by $L^2$-OCD. The dynamics is illustrated in Fig.~\ref{fig:clustering}, where the  correlation between the data points monotonically increases in time. The Monge map, given by identity,  emerges from particle clusters as time progresses.
\item {\bf Symmetric Positive Semi-Definite  Cone.} As  a by-product of OCD, we justified in Proposition \ref{prop:spd}  that the dynamics ensures the cross-correlation matrix remains in the cone of symmetric positive semi-definite matrices. Here we verify this finding by numerical simulation, where the initial condition is product measure of two centered Gaussians with covariance matrices $\Sigma_\mu$ and $\Sigma_\nu$. As  shown in Fig.~\ref{fig:psd_cone}, OCD evolves the cross-correlation inside the cone and reaches the maximum on the curve connecting $\Sigma_\mu$ and $\Sigma_\nu$. Note that there exist interesting results on the geometry of the cross-correlation of Gaussians induced by the $L^2$-optimal transport and its connection to geometric average of $\Sigma_\mu$ and $\Sigma_\nu$, which we leave out of discussion for brevity, but refer the reader to \cite{bhatia2019bures}. 
\item {\bf Nonlinear Monge Map.} To demonstrate relevance of OCD in recovering nonlinear multi-dimensional maps, we consider a synthetic setting where the target distribution is chosen such that  the Monge map is known analytically. Consider the standard normal distribution over $\mathbb{R}^2$, denoted by $\mu$. Let the target distribution be the push-forward by softmax map $\nu=T\#\mu$, with $T=\nabla_x \log(\exp(x_1)+\exp(x_2))$. Since $T$ is gradient of a convex function, the Monge map between $\mu$ and $\nu$ is simply $T$. Depicted in Fig.~\ref{fig:softmax}, starting from independent $10^6$ samples of $\mu$ and $\nu$, particles pushed by OCD with $\epsilon=10^{-3}$ and $\Delta t=10^{-2}$ successfully recover the Monge map and accurately estimate 2-Wasserstein distance, as $t$ progresses. Note that the underestimation of the 2-Wasserstein distance is due to the fact that the marginals may not be preserved due to different sources of numerical errors, leading to, e.g. smaller variances of each marginal.
\end{enumerate}

\subsection{Distribution Learning and Sampling}
\label{sec:density_estimation}
\noindent We deploy the proposed OCD (with $L^2$-cost) to learn transport map between the normal source and a target density $\in$ $\{\textrm{Banana}, \textrm{Funnel}, \textrm{Swiss roll}\}$ \cite{baptista2023representation}. We use $10^4$ samples of each marginal and carry out OCD to find the samples of the optimal joint density on $(\hat X^*,\hat Y^*)$. Afterwards, in order to learn the map via the paired data points, i.e. $M_{X\rightarrow Y}:\hat X^* \rightarrow \hat Y^*$, we deploy NN with $4$ layers, each with $100$ neurons, 
where Adam's algorithm \cite{diederik2014adam} with a learning rate of $0.002$ and $10^4$ iterations are utilized to find the NN weights.
\\ \ \\
For testing, we generate $10^6$ normally distributed samples, i.e. $X^\text{test} \sim \mathcal{N}( 0, I_{n\times n})$ and feed them to the NN to find $M_{X \rightarrow Y}(X^\text{test})$. As shown in Fig.~\ref{fig:density_est}, OCD is capable of recovering the target density with a reasonable accuracy. However, we note that the estimated densities by OCD are slightly smoothed out compared to the ground-truth. This error could be related to the current approach of estimating the conditional expectation, which resembles Kernel Density Estimation with known smoothing effects.  In Supplementary Information, we provide a study of hyper-parameters and comparison with Adaptive Transport Maps (ATM) \cite{baptista2023representation} for each test case. 
We point out that once the map is learned, computing the new samples is straight-forward, given that the network only performs simple matrix-vector multiplication on the samples of the standard normal distribution. Hence the proposed workflow can efficiently serve as a sampling approach for complex distributions (such as Swiss roll), with ramifications for sampling algorithms based on measure transport \cite{sun2024dynamical,papamakarios2021normalizing}.
\begin{figure}
  	\centering
   \begin{tabular}{ccccc}
   \includegraphics[scale=0.3]{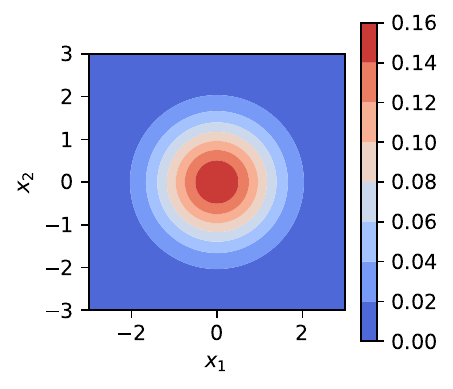}
    &
   \rotatebox{90}{\hspace{0.5cm} True}
   &
   \includegraphics[scale=0.3]{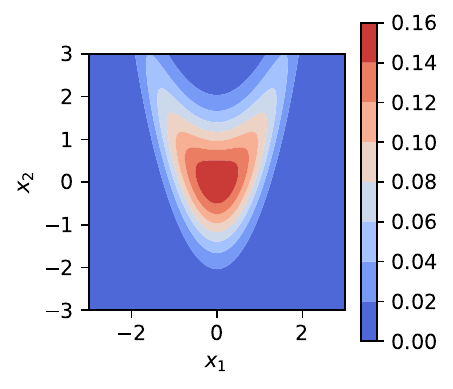}
   &
\includegraphics[scale=0.33]{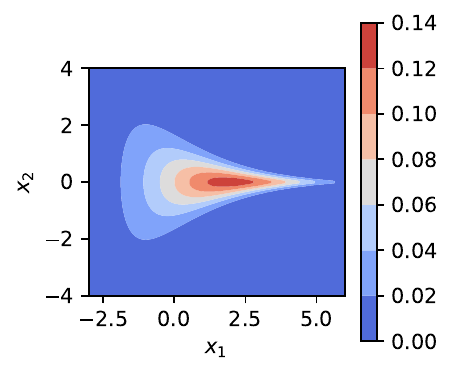}
&
\includegraphics[scale=0.3]{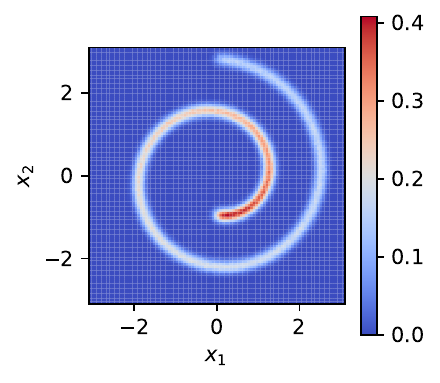}
\\
& \includegraphics[scale=0.25]{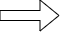}
\\
\includegraphics[scale=0.30]{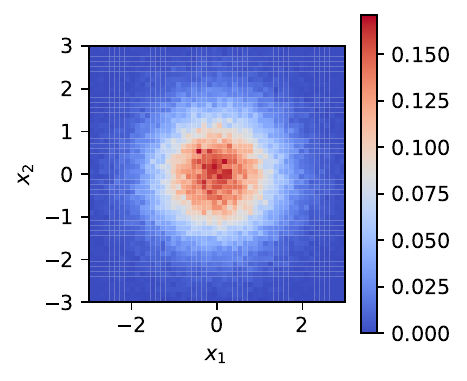}
&
\rotatebox{90}{\hspace{0.5cm} OCD}
&
\includegraphics[scale=0.3]{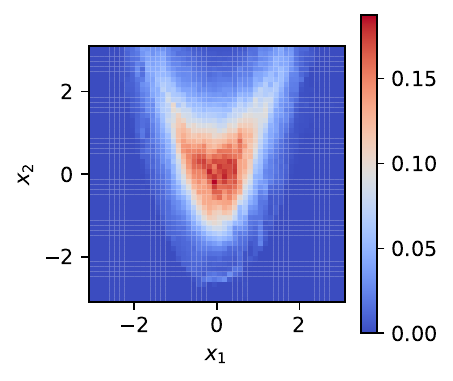}
&
\includegraphics[scale=0.33]{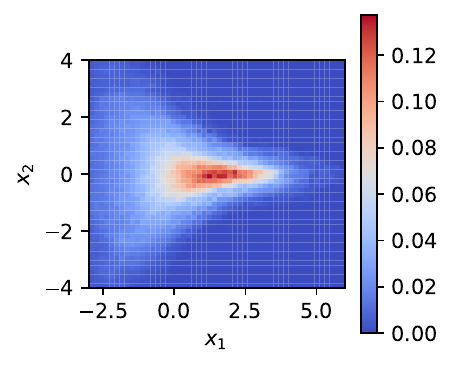}
&
\includegraphics[scale=0.3]{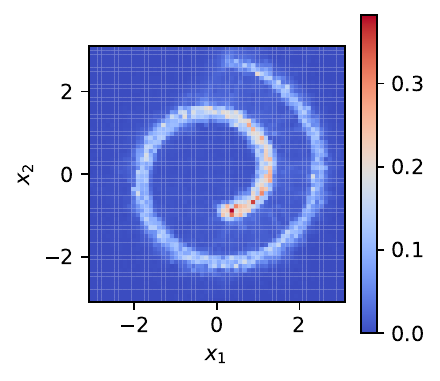}
\\
Normal & & Banana & Funnel & Swiss roll
   \end{tabular}
   \caption{Estimating density of Banana, Funnel, and Swiss roll distributions using OCD. Here, first we find optimal pairs among $10^4$ samples of normal and target density using OCD. We train NN on the sorted data and generate new samples of the target density using the learned map and $10^6$ normally distributed particles. The true marginals are depicted on top and histogram of estimated marginals on bottom.}
      \label{fig:density_est}
\end{figure}

\subsection{Color Interpolation}

\noindent As the final example, we consider another application of optimal transport in interpolation between color distribution of different images. Given an image with dimension $(N_x,N_y,3)$, the pixel colors can be considered as the random variable, leading to $N_x\times N_y$ samples for $X$ with $\dim(X)=3$, indicating red/green/blue channels. We take a low resolution picture of the ocean taken in day and sunset time for training, as depicted in Fig.~\ref{fig:tr_ocean} (see examples of POT library \cite{JMLR:v22:20-451} for details).
\begin{figure}
  	\centering
   \begin{tabular}{cc}
\includegraphics[scale=0.3]{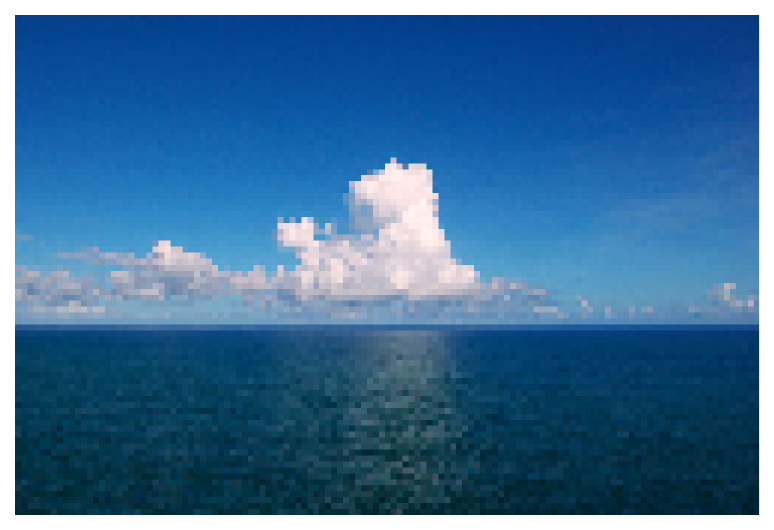}
&
\includegraphics[scale=0.3]{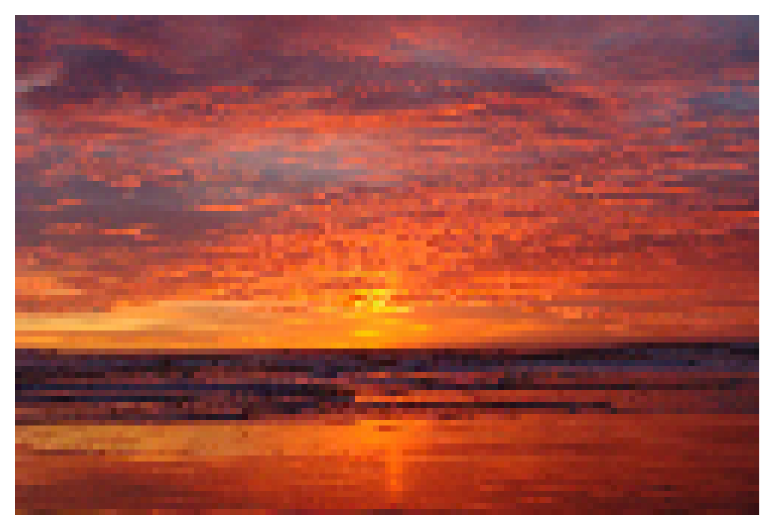}
\\
$X_0$ & $Y_0$
\end{tabular}
\caption{Low resolution ($84\times 125
$ pixels) pictures of ocean in day and sunset time used for training.}
\label{fig:tr_ocean}
\end{figure}
\noindent Given $N_p=10'500$ training samples of these two pictures, we learn the map $M_{X\rightarrow Y}$ using paired samples found from OCD (with $L^2$-cost) and training NN similar to the one described in \cref{sec:density_estimation}. We employ $\Delta t=0.1$ and $\epsilon=\epsilon_{\textrm{max}}(N_{\textrm{clusters}}/N_p>0.9)$. We repeat the same procedure for EMD to benchmark, and found that while EMD takes about $2.6$ GB of memory, OCD only consumes $115$ MB.  Once the NN map is learned, we test the map by feeding the NN with the full resolution of day picture with $670\times 1000$ pixels, where we see a good agreement with the EMD solution.  We further test the learned OCD map on two new testing picture with 
$548\times 1024$ and $1200\times 1920$ pixels. As shown in Fig.~\ref{fig:interpolate_color_ocean}, the trained OCD map is capable of transforming the testing picture into sunset. 

\subsection{Further numerical studies}
In Supporting Information, we have provided more detailed numerical studies for the reader as a reference. This includes error analysis with respect to the OCD's free parameters, and further comparisons against other benchmark methods including the invertible map \cite{baptista2023representation} and Neural optimal transport \cite{korotin2023neural}. Furthermore, we show solution of the OCD with the $L^p$-cost function as well as OT map between two 10-dimensional multivariate normal distributions.

\begin{figure}
  	\centering
   \begin{tabular}{cc}
   \rotatebox{90}{\ \ \ \ \ OCD} & 
\includegraphics[scale=0.58]{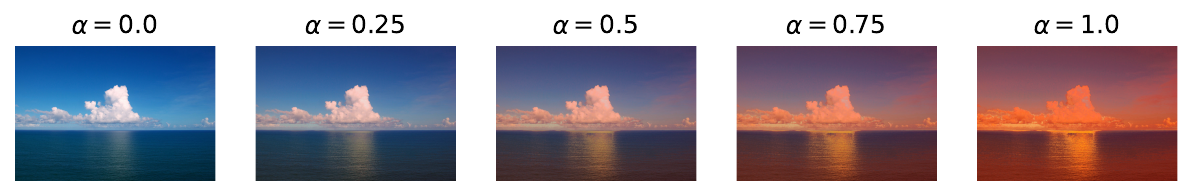}
\\
\rotatebox{90}{\ \ \ \ \ EMD} & 
\includegraphics[scale=0.58, trim={0 0 0 0.6cm}, clip]{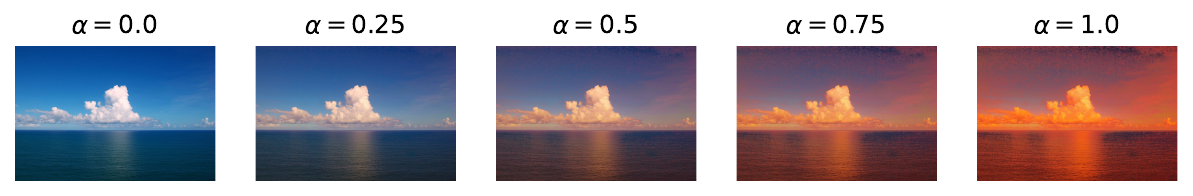}
\\
\rotatebox{90}{\hspace{-1.5cm} Test \#1: OCD}
&
\includegraphics[scale=0.58, trim={0 0 0 0.6cm}, clip]{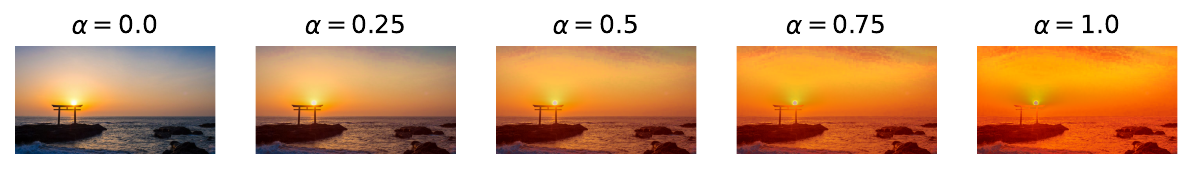}
\\
\rotatebox{90}{\ \ \ \ \ }
&
\includegraphics[scale=0.56, trim={0 0 0 0.6cm}, clip]{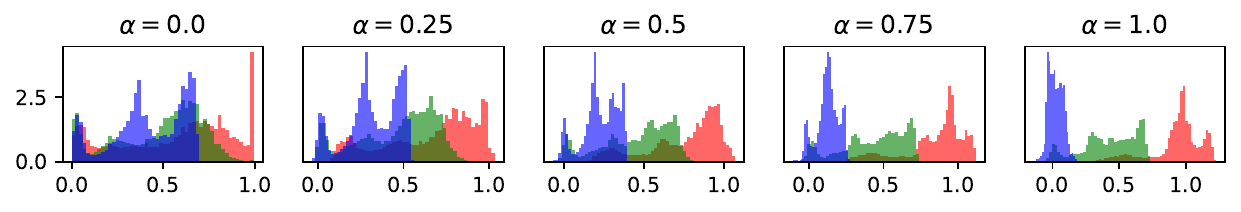}
\\
\rotatebox{90}{\hspace{-1.5cm} Test \#2: OCD}
&
\includegraphics[scale=0.58, trim={0 0 0 0.6cm}, clip]{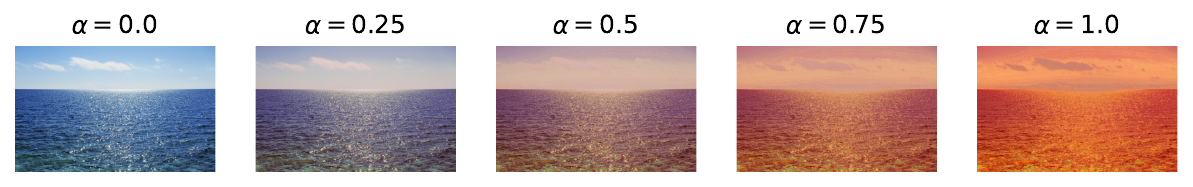}
\\
\rotatebox{90}{\ \ \ \ \ }
&
\includegraphics[scale=0.56, trim={0 0 0 0.6cm}, clip]{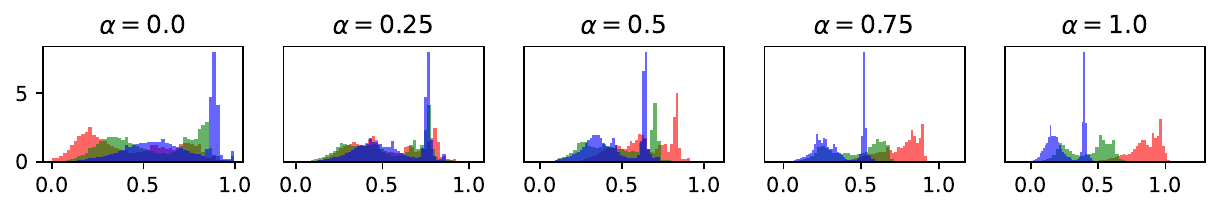}
\end{tabular}
\caption{Interpolating day and sunset pictures of ocean. The interpolation is done using the learned map $M_{X\rightarrow Y}(.)$ via $(1-\alpha)X + \alpha M_{X\rightarrow Y}(X)$. The bottom two rows show the transport of colors for two test picture.  The histograms show the distribution of red, green, and blue colors for the corresponding image.
}
\label{fig:interpolate_color_ocean}
\end{figure}

\section{Discussion} 
\subsection{Concluding Remarks}
 Conceptually, what makes the OCD approach particularly appealing is that it is at cross-road of different topics: 
\begin{enumerate}
\item The original problem of Monge-Kantorovich optimization, constrained on a set of couplings, is reduced to a dynamics whose evolution is controlled by the conditional expectation. The latter is the solution of the regression problem, and thereby, the OCD provides a direct link between Bayesian regression and the optimal transportation. 
\item Using basic interaction rules, the dynamics gives a particle-type relaxation of independent samples, seeking their optimal coupling. The notion of clustering, which arises from this coupling, resonates with models of opinion dynamics and offers a method for obtaining sparse representations of complex, high-dimensional datasets.
\item  The OCD framework has potential applications as the foundation for generative models. When combined with supervised learning, the learned maps could be employed to generate new datasets.
\item Finally, the dynamics of OCD can be interpreted as a form of projected gradient descent, providing a novel approach through which optimization techniques constrained on a set of distributional coupling can be constructed.
\end{enumerate}
From numerical point of view, the model has favorable scaling behaviour. In the most basic form, it scales nearly linear with number of data points as well as dimension of distributional support, making OCD competitive in comparison to the state-of-the-art random map algorithms as well as direct linear programming approach. This powerful computational efficiency enables us to treat distributional learning problems over more complex and higher dimensional settings. \\ \ \\
A crucial structural property of OCD, initialized by independent samples, is that the
induced correlations remain symmetric and positive semidefinite. 
While this constitutes an interesting feature for randomized transport maps, it 
also implies that OCD, in its current deterministic formulation, does not in 
general recover the optimal coupling whenever the latter exhibits intrinsic 
asymmetries in its transport structure.  In this sense, the symmetry constraint acts as an implicit bias, different in nature from the entropic bias induced 
by Schrödinger bridge or Sinkhorn-type methods 
(e.g.\ \cite{leonard2013survey,cuturi2013sinkhorn,chen2021stochastic}). 
Whereas entropic algorithms favor couplings close to product measures, the OCD 
dynamics bias the evolution toward symmetric correlation structures.\\ \ \\
The instability of suboptimal stationary 
solutions suggests a natural direction for relaxing this constraint: introducing 
controlled randomization into the dynamics, to explore the coupling space 
more flexibly. We anticipate that stochastic perturbations, e.g. in the form of annealing of the projected Langevin dynamics \cite{conforti2023projected} 
may allow OCD-type flows to escape symmetric couplings and reach more general 
stationary couplings. A systematic theoretical characterization of the marginal 
distributions and structural conditions under which OCD recovers the optimal 
transport map remains an open research direction.

\subsection{Limitations} The current treatment of conditional expectation exhibits diffusive behaviour which can deteriorate  marginal distributions with sharp features. More elaborated approaches, e.g. deployment of Gaussian Processes or Neural Nets as hypothesis class of the conditional expectation, may alleviate this issue, at least to some degree, though likely at the cost of increased computational complexity. 
In addition, moment constraints of the marginal distributions can be leveraged to enforce the marginal densities of OCD to remain close to the initial condition, and hence reducing the numerical dispersion. 
Another issue that requires further investigation is the optimal choice of $\epsilon$. Although we carried out extensive numerical tests, see Supplementary Information, we believe computing the conditional expectation efficiently is a separate task by itself.
These trade-offs between accuracy and efficiency are interesting areas for future research on OCD.  
From a theoretical standpoint, our understanding of the underlying dynamics remains incomplete. While we provided some important theoretical results on OCD, questions such as the well-posedness of OCD and its convergence towards optimal solution are still unresolved and require further exploration. 
\subsection{Broader Context} The presented dynamics, along with devised non-parametric algorithm, gives a means to pair up distinct data-sets, through a dynamics whose increment is a functional of the distribution. This is a prototype of McKean-Vlasov processes, which offer further generalizability compared to the conventional regression, due its dependency on the distribution. We believe this development holds implications which could go beyond the context of optimal transport, specifically given recent attention towards McKean-Vlasov universal approximators emerging in data-driven trainable models \cite{yang2024neural,atanackovic2024meta}. 

\section{Acknowledgments}
MS thanks Andreas Adelmann for his support and constructive feedback.
PME was supported
by the ERC grant TRUST-949796 and the NSERC Discovery grant RGPIN-2025-06544.

\bibliographystyle{siam}
\bibliography{refs}

\appendix

\section{Numerical Scheme and Solution Algorithm}
\label{sec:algo}
\noindent The algorithmic steps of OCD consist of estimating the conditional expectation, along with an appropriate time integration scheme. Explicit formulas are given for both in the following. 

\subsection {Estimation of Conditional Expectation}
\begin{enumerate}
\item Piecewise constant estimation follows simple averaging per cluster, which gives
\begin{eqnarray}
\hat{\mathcal{K}}_X(\hat{X}_i^n|\hat{X}^n,\hat{Y}^n)&=&\frac{1}{N_{\bar{I}^{\hat{X}_i^{n},\epsilon}}}\sum_{j\in \bar{I}^{\hat{X}_i^{n},\epsilon}}\nabla_x c(\hat{X}^n_i,\hat{Y}^n_j) \label{eq:cond-const-x} \\
\textrm{and} \ \ \
\hat{\mathcal{K}}_Y(\hat{Y}_i^n|\hat{X}^n,\hat{Y}^n)&=&\frac{1}{N_{\bar{I}^{\hat{Y}_i^{n},\epsilon}}}\sum_{j\in \bar{I}^{\hat{Y}_i^{n},\epsilon}}\nabla_y c(\hat{X}^n_j,\hat{Y}^n_i) \label{eq:cond-const-y}\ .
\end{eqnarray}
\item For the piecewise linear estimation, first, let us define the average values per cluster
\begin{eqnarray}
\hat{m}^n_X=\frac{1}{N_{\bar{I}^{\hat{X}_i^{n},\epsilon}}}\sum_{j\in \bar{I}^{\hat{X}_i^{n},\epsilon}}\nabla_x c(\hat{X}^n_j,\hat{Y}^n_j) \ \ \ \textrm{and} \ \ \ \hat{m}^n_Y=\frac{1}{N_{\bar{I}^{\hat{Y}_i^{n},\epsilon}}}\sum_{j\in \bar{I}^{\hat{Y}_i^{n},\epsilon}}\nabla_x c(\hat{X}^n_j,\hat{Y}^n_j) \ . \nonumber
\end{eqnarray}
Accordingly, we get the correlation matrices
\begin{eqnarray}
\hat{\Sigma}^n_{XX}&=&\frac{1}{N_{\bar{I}^{\hat{X}_i^{n},\epsilon}}}\sum_{j\in \bar{I}^{\hat{X}_i^{n},\epsilon}}\left(\nabla_x c(\hat{X}^n_j,\hat{Y}^n_j)-\hat{m}^n_X\right)\otimes \left(\nabla_x c(\hat{X}^n_j,\hat{Y}^n_j)-\hat{m}^n_X\right)\ ,  \nonumber \\
\hat{\Sigma}^n_{YY}&=&\frac{1}{N_{\bar{I}^{\hat{Y}_i^{n},\epsilon}}}\sum_{j\in \bar{I}^{\hat{Y}_i^{n},\epsilon}}\left(\nabla_y c(\hat{X}^n_j,\hat{Y}^n_j)-\hat{m}^n_Y\right)\otimes \left(\nabla_y c(\hat{X}^n_j,\hat{Y}^n_j)-\hat{m}^n_Y\right)\ ,  \nonumber \\
\hat{\Sigma}^n_{XY}&=&\frac{1}{N_{\bar{I}^{\hat{X}_i^{n},\epsilon}}}\sum_{j\in \bar{I}^{\hat{X}_i^{n},\epsilon}}\left(\nabla_x c(\hat{X}^n_j,\hat{Y}^n_j)-\hat{m}^n_X\right)\otimes \left(\nabla_y c(\hat{X}^n_j,\hat{Y}^n_j)-\hat{m}^n_Y\right)  \nonumber \\
\textrm{and} \ \ \ \hat{\Sigma}^n_{YX}&=&\frac{1}{N_{\bar{I}^{\hat{Y}_i^{n},\epsilon}}}\sum_{j\in \bar{I}^{\hat{Y}_i^{n},\epsilon}}\left(\nabla_x c(\hat{X}^n_j,\hat{Y}^n_j)-\hat{m}^n_X\right)\otimes \left(\nabla_y c(\hat{X}^n_j,\hat{Y}^n_j)-\hat{m}^n_Y\right),  \nonumber
\end{eqnarray}
leading to
\begin{eqnarray}
\hat{\mathcal{K}}_X(\hat{X}_i^n|\hat{X}^n,\hat{Y}^n)&=&\hat{m}^n_Y+(\hat{\Sigma}_{XY}^n)^T(\hat{\Sigma}_{XX}^n)^{-1}(\hat{X}^n_i-\hat{m}^n_X)  \nonumber \\
\textrm{and} \ \ \ \hat{\mathcal{K}}_Y(\hat{Y}_i^n|\hat{X}^n,\hat{Y}^n)&=&\hat{m}^n_X+(\hat{\Sigma}_{YX}^n)^T(\hat{\Sigma}_{YY}^n)^{-1}(\hat{Y}^n_i-\hat{m}^n_Y)  \nonumber \ .
\end{eqnarray}
If the number of particles in a cluster drop below two, the correlation matrix becomes singular. In order to ensure that the matrices $\hat{\Sigma}_{XX,YY}^n$ remain invertible, they are regularized by $\hat{\epsilon}>0$, as the following 
\begin{eqnarray}
\hat{\mathcal{K}}^{\hat{\epsilon}}_X(\hat{X}_i^n|\hat{X}^n,\hat{Y}^n)&=&\hat{m}^n_Y+(\hat{\Sigma}_{XY}^n)^T(\hat{\Sigma}_{XX}^n+\hat{\epsilon}I_X)^{-1}(\hat{X}^n_i-\hat{m}^n_X) \label{eq:cond-linear-x} \nonumber \\
 && \\
\textrm{and} \ \ \ \hat{\mathcal{K}}^{\hat{\epsilon}}_Y(\hat{Y}_i^n|\hat{X}^n,\hat{Y}^n)&=&\hat{m}^n_X+(\hat{\Sigma}_{YX}^n)^T(\hat{\Sigma}_{YY}^n+\hat{\epsilon}I_Y)^{-1}(\hat{Y}^n_i-\hat{m}^n_Y) \label{eq:cond-linear-y}, \nonumber \\
\end{eqnarray}
where $I_{X,Y}$ are identity matrices with the same size as $\hat{\Sigma}_{XX,YY}^n$, respectively \cite{thrampoulidis2015regularized}.   
\end{enumerate}
\subsection {Time Discretization}
For the fourth-stage Runge-Kutta, we evaluate the slopes
\begin{eqnarray}
{k^X_{1,i}}&=&\nabla_x c (\hat{X}_i^{n},\hat{Y}_i^{n})-\hat{\mathcal{K}}_X(\hat{X}_i^n|\hat{X}^n,\hat{Y}^n)\label{eq:st1-x}  \\
\textrm{and} \ \ \ 
k^Y_{1,i}&=&\nabla_y c (\hat{X}_i^{n},\hat{Y}_i^{n})-\hat{\mathcal{K}}_Y(\hat{Y}_i^n|\hat{X}^n,\hat{Y}^n)   \label{eq:st1-y}
\end{eqnarray}
at the beginning,
\begin{eqnarray}
k_{2,i}^X&=&\nabla_x c (\hat{X}_i^{k_1},\hat{Y}_i^{k_1})-\hat{\mathcal{K}}_X(\hat{X}_i^{k_1}|\hat{X}^{k_1},\hat{Y}^{k_1}) \label{eq:st2-x} \\
\textrm{and} \ \ \ 
k_{2,i}^Y&=&\nabla_y c (\hat{X}_i^{k_1},\hat{Y}_i^{k_1})-\hat{\mathcal{K}}_Y(\hat{Y}_i^{k_1}|\hat{X}^{k_1},\hat{Y}^{k_1})\, \label{eq:st2-y}
\end{eqnarray}
with $\hat{X}_i^{k_1}=\hat{X}_i^n+k_{1,i}^X{\Delta t}/{2}$ and $\hat{Y}_i^{k_1}=\hat{Y}_i^n+k_{1,i}^Y{\Delta t}/{2}$, at the first mid-point,
\begin{eqnarray}
k_{3,i}^X&=&\nabla_x c (\hat{X}_i^{k_2},\hat{Y}_i^{k_2})-\hat{\mathcal{K}}_X(\hat{X}_i^{k_2}|\hat{X}^{k_2}, \hat{Y}^{k_2})  \label{eq:st3-x} \\
\textrm{and} \ \ \ k_{3,i}^Y&=&\nabla_y c(\hat{X}_i^{k_2},\hat{Y}_i^{k_2})-\hat{\mathcal{K}}_Y(\hat{Y}_i^{k_2}|\hat{X}^{k_2}, \hat{Y}^{k_2}) \, \label{eq:st3-y}
\end{eqnarray}
with $\hat{X}_i^{k_2}=\hat{X}_i^n+k_{2,i}^X{\Delta t}/{2}$ and $\hat{Y}_i^{k_2}=\hat{Y}_i^n+k_{2,i}^Y{\Delta t}/{2}$, at the second mid-point, and
\begin{eqnarray}
k_{4,i}^X&=&\nabla_x c (\hat{X}_i^{k_3},\hat{Y}_i^{k_3})-\hat{\mathcal{K}}_X(\hat{X}_i^{k_3}|\hat{X}^{k_3}, \hat{Y}^{k_3}) \label{eq:st4-x} \\
\textrm{and} \ \ \ k_{4,i}^Y&=&\nabla_y c (\hat{X}_i^{k_3},\hat{Y}_i^{k_3})-\hat{\mathcal{K}}_Y(\hat{Y}_i^{k_3}|\hat{X}^{k_3}, \hat{Y}^{k_3}) \, \label{eq:st4-y}
\end{eqnarray}
with $\hat{X}_i^{k_3}=\hat{X}_i^n+k_{3,i}^X{\Delta t}$ and $\hat{Y}_i^{k_3}=\hat{Y}_i^n+k_{3,i}^Y{\Delta t}$, at the end-point, where for notation brevity we omit the dependency on the regularizer $\hat{\epsilon}$. The update for full $\Delta t$ then is given by weighted average of all four slopes
\begin{eqnarray}
\hat{X}_i^{n+1}&=&\hat{X}_i^n+\frac{1}{6}\left(k_1^X+2k_2^X+2k_3^X+k_4^X\right) \\
\textrm{and} \ \ \ \hat{Y}_i^{n+1}&=&\hat{Y}_i^n+\frac{1}{6}\left(k_1^Y+2k_2^Y+2k_3^Y+k_4^Y\right) \ .
\end{eqnarray}
See \cref{alg:RK4_OCD} for summary of $L^2$-OCD computational step. 

\begin{algorithm}
 \caption{Main steps of the OCD algorithm for $L^2$-cost, using piecewise linear estimation of the conditional expectation and fourth-order Runge-Kutta.}
   \label{alg:RK4_OCD}
\begin{algorithmic}
\State{Input $N_p$ samples of $(\hat{X}^0,\hat{Y}^0)$ }
\State{Set $\beta=0.9$, $\epsilon=\max(\epsilon| N_{\textrm{clusters}}/N_p>\beta) $ and $\hat{\epsilon}=0$}
\State{Set convergence threshold $\gamma=0.01$}
\State{Set $t_f,n_t$ and time step size $\Delta t=t_f/n_t$ (default value $\Delta t=0.1$)}
\While{$\mathbb{E}[||\hat{X}^n-\hat{Y}^n||^2_2]> \gamma$}
   \State{$\hat{X}^{k_0} \leftarrow \hat{X}^n$ }
    \State{$\hat{Y}^{k_0} \leftarrow \hat{Y}^n$}     
    \For {$j=0:2$}
        \State{$ \textrm{indX} \leftarrow \textrm{BallTree}(\hat{X}^{k_j},\epsilon)$}  
        \State{$ \textrm{indY} \leftarrow \textrm{BallTree}(\hat{Y}^{k_j},\epsilon)$}    
         \For{$i=1:N_p$}
             \State{$\bar{I}^{\hat{X}^{k_j}_i, \epsilon}\leftarrow \textrm{indX}[i]$}
            \State{$\bar{I}^{\hat{Y}^{k_j}_i, \epsilon}\leftarrow \textrm{indY}[i]$} 
            \State{Estimate $\mathbb{E}[\nabla_x c(\hat{X}^{k_j},\hat{Y}^{k_j})|\hat{X}^{k_j}_i]$ using Eq.~\eqref{eq:cond-linear-x}}
            \State{Estimate $\mathbb{E}[ \nabla_y c(\hat{X}^{k_j},\hat{Y}^{k_j})|\hat{Y}^{k_j}_i]$ using Eq.~\eqref{eq:cond-linear-y}}
            \State{Compute $k^{X}_{j+1,i}$ using Eqs.~\eqref{eq:st1-x}, \eqref{eq:st2-x}, \eqref{eq:st3-x}}
            \State{Compute $k^{Y}_{j+1,i}$ using Eqs.~\eqref{eq:st1-y}, \eqref{eq:st2-y}, \eqref{eq:st3-y}}
            \State{$\hat{X}^{k_{j+1}}_i=\hat{X}_i^n+k_{j+1,i}^X\Delta t/w[j]$}
            \State{$\hat{Y}^{k_{j+1}}_i=\hat{Y}_i^n+k_{j+1,i}^Y\Delta t/w[j]$}
         \EndFor
       \State{Compute $k^{X}_{4,i}$ using Eq.~\eqref{eq:st4-x}}
        \State{Compute $k^{Y}_{4,i}$ using Eq.~\eqref{eq:st4-y}} 
    \EndFor
    \State{$\hat{X}^{n+1}\leftarrow\hat{X}^n+(k_1^X+2k_2^X+2k_3^X+k_4^X)/6$}
    \State{$\hat{Y}^{n+1}\leftarrow \hat{Y}^n+(k_1^Y+2k_2^Y+2k_3^Y+k_4^Y)/6$}
     \State{Increment $n$}
  \EndWhile
\end{algorithmic}
\end{algorithm}


\section{Optimal Cutoff}
\label{sec:find_opt_eps}
\noindent We perform numerical experiments to investigate the optimal value of $\epsilon$ by considering the transport map between two normal distributions $\mu = \mathcal{N}(0,I)$ and  $\nu=\mathcal{N}(1,I)$. In particular, given $N_p$ particles of these two marginals, we run the $L^2$-OCD algorithm until convergence, for a range of $\epsilon$ in one and two dimensional probability space, and analyse our finding. In all experiments here, we consider $N_p\in\{100, 200, 400, 800, 1'600 \}$ and $\Delta t = 10^{-3}$, where maximum of $20'000$ steps are employed. Furthermore, both piecewise linear and piecewise constant estimation of conditional expectations are used to characterize $\epsilon$.
\\ \ \\
While the Monge map is a simple affine transformation in this setting, we use a numerical solution of the linear programming, noted here by EMD \cite{bonneel2011displacement}, to understand the quality of OCD results with respect to the EMD benchmark. To find the optimal value denoted by $\epsilon^*$, from several experiments of OCD with a range of $\epsilon$,  we compare the solution of OCD with EMD by computing the Wasserstein distance of their joint densities. Therefore, we construct the discrete joint density by concatenation $(X,Y)\sim \pi$ and compute $d^2(\pi^\mathrm{EMD},\pi^\mathrm{OCD})$. Furthermore, we study the number of clusters created for each OCD run, using Density-Based Spatial Clustering of Applications with Noise  (DBSCAN) method \cite{schubert2017dbscan}.
\\ \ \\
From Figs.~\ref{fig:err_np_eps_normal_no_cluster}-\ref{fig:err_np_eps_normal_no_cluster_linear}, it is clear that for small $\epsilon$, i.e. $\epsilon\rightarrow0$, the number of clusters does not differ much from the number of particles $N_p$. In this case, the initial particles remain frozen and do not see each other, hence the OCD algorithm returns the initial joint density (which is density of the product measure). However, for large values of $\epsilon$, the number of clusters decreases to a fixed value depending on $N_p$, dimension of $X$ and $Y$, while the Wasserstein distance of OCD from the solution of EMD increases exponentially (for piecewise constant approximation of conditional expectation). Note that we observe a decay in error for large $\epsilon$, if linear approximation of conditional expectation is employed. This can be explained by the fact that linear regression is indeed optimal if conditional expectation with respect to Gaussian measure is considered. Since here both marginals are Gaussians, we observe an improved performance of the piecewise linear approximation, for very large $\epsilon$.  Nevertheless, in all considered scenarios, the optimal parameter $\epsilon^*$ lies somewhere between two extremes of $\epsilon\to 0$ and $\epsilon \to \infty$.

\begin{figure}
  	\centering
   \begin{tabular}{cc}
   \includegraphics[scale=0.5]{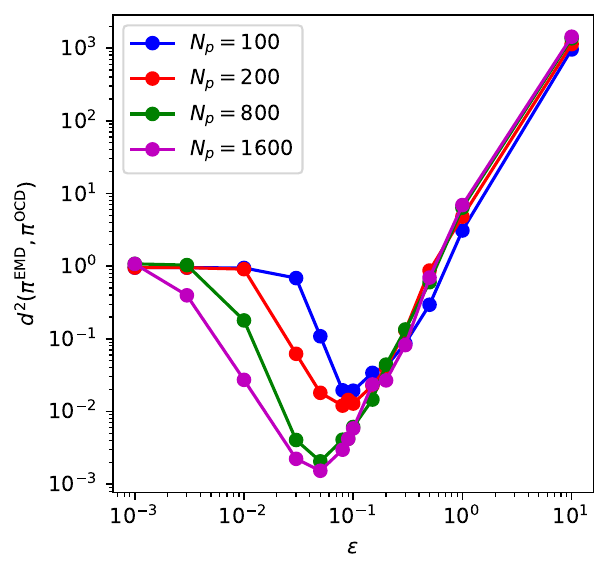}
   &
\includegraphics[scale=0.5]{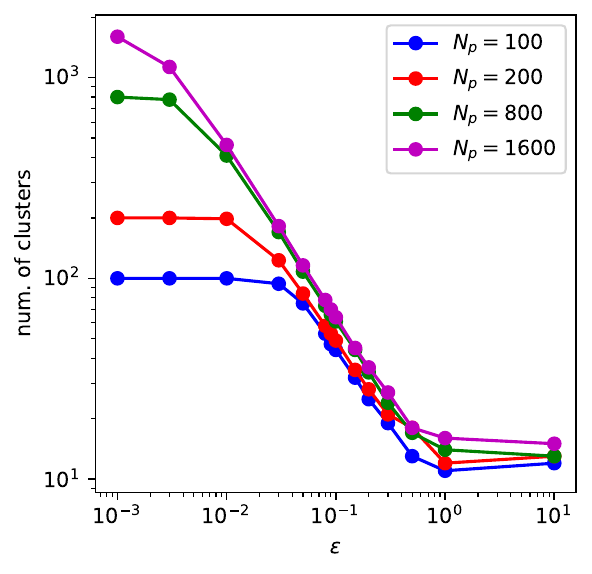}
\\
   \includegraphics[scale=0.5]{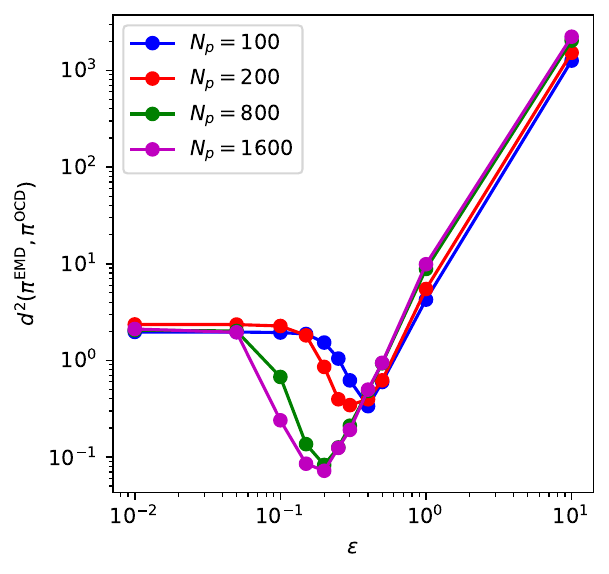}
   &
\includegraphics[scale=0.5]{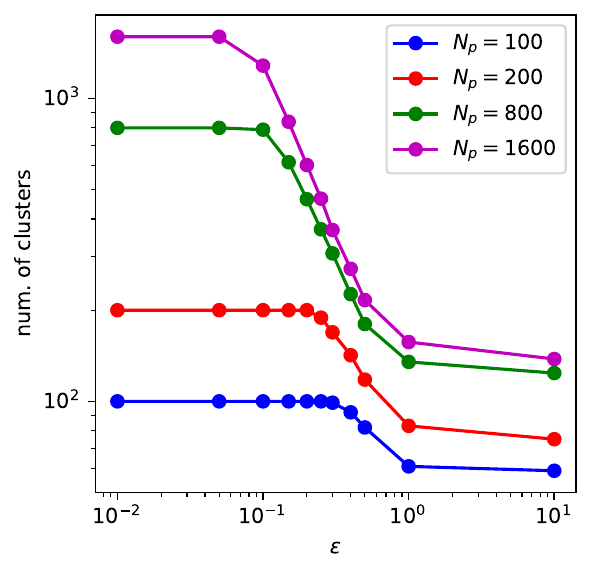}
   \end{tabular}
   \caption{Error in finding the optimal map between marginals $\mu=\mathcal{N}( 0, I)$ and $\nu=\mathcal{N}( 1, I)$ as well as the number of clusters in 1-dimensional space, i.e. $\dim(X)=1$ (top) and 2-dimensional space, i.e. $\dim(X)=2$ (bottom). The piecewise constant approximation is used to estimate the conditional expectation.}
      \label{fig:err_np_eps_normal_no_cluster}
\end{figure}

\begin{figure}
  	\centering
   \begin{tabular}{cc}
   \includegraphics[scale=0.5]{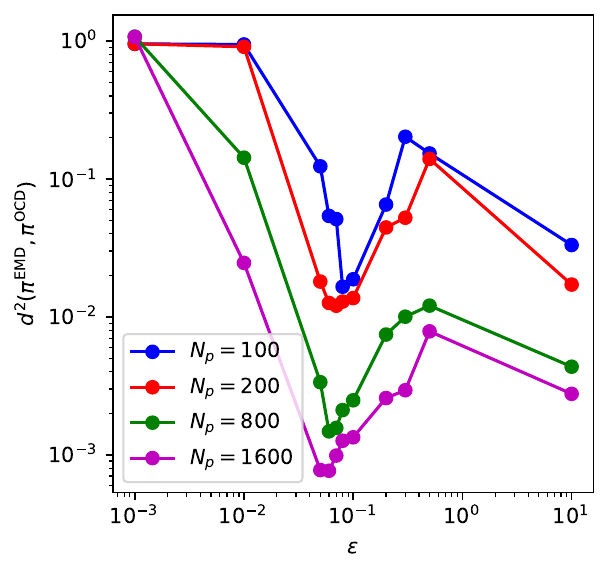}
   &
\includegraphics[scale=0.5]{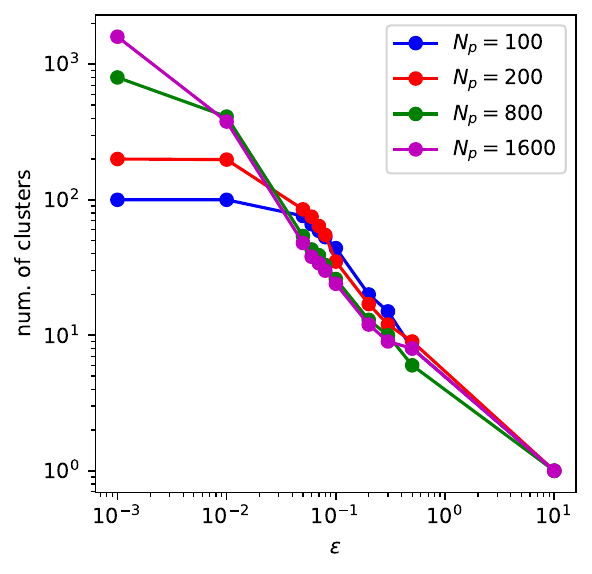}
\\
\includegraphics[scale=0.5]{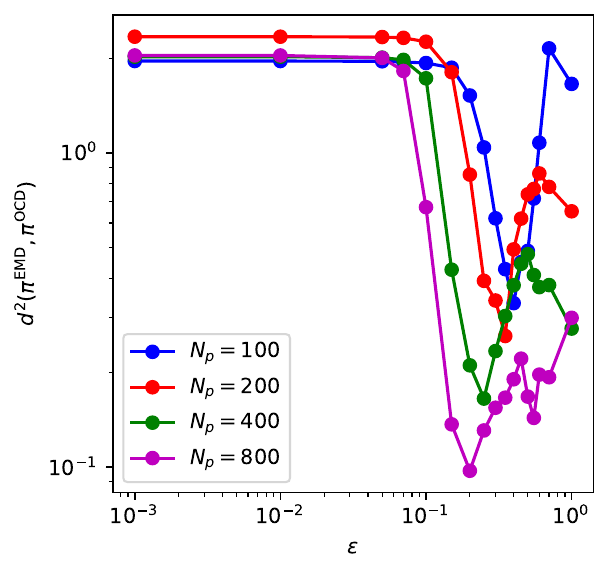}
   &
\includegraphics[scale=0.5]{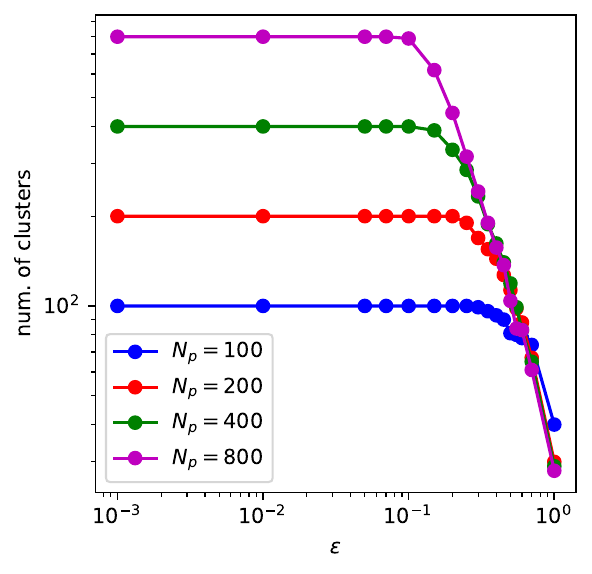}
   \end{tabular}
   \caption{Error in finding the optimal map between marginals $\mu=\mathcal{N}( 0, I)$ and $\nu=\mathcal{N}( 1, I)$, as well as the number of clusters in 1-dimensional space, i.e. $\dim(X)=1$ (top) and 2-dimensional space, i.e. $\dim(X)=2$ (bottom). The piecewise linear approximation is considered for the estimation of conditional expectation.}
      \label{fig:err_np_eps_normal_no_cluster_linear}
\end{figure}

\ \\
In \cref{fig:opt_eps_vs_Np}, we plot the value of optimal parameter $\epsilon^*$ against $N_p$. Interestingly, in both cases of $\dim(X)=\dim(Y)=1$ and $2$, we observe that the optimal value scales almost $\epsilon^*\sim N_p^{-1/5}$ up to statistical noise. Since our proposed approach in computing the conditional expectation (piecewise constant) can be seen as a kernel density estimator with a step function as the kernel, the scaling of the optimal bandwidth $N_p^{-1/5}$ is expected. 
\\ \ \\
Another insight can be gained by considering number of clusters as a guiding proxy for $\epsilon^*$. We consider two proxies that could relate $\epsilon^*$ to the number of clusters:
\begin{enumerate}
    \item $\epsilon_\mathrm{crit}$: the critical point of number of clusters versus $\epsilon$.
    \item $\epsilon_\mathrm{max}$: the maximum $\epsilon$ with $N_\mathrm{clusters}/N_p > \beta$, with $\beta$ is a user-defined tolerance.
\end{enumerate}
Numerical results presented in Figs.~\ref{fig:crit_max_versus_opteps}-\ref{fig:crit_max_versus_opteps_lin} motivate that both proxies lead to reasonable approximations of $\epsilon^*$. While there is no clear consensus on all the considered numerical tests, we note that the optimal $\epsilon^*$ is typically close to one of the mentioned proxies. Therefore, the user can preform a grid search around either $\epsilon_\text{crit}$ or $\epsilon_\text{max}$.
While presented proxies provide some guidance, we emphasize that, in general, the optimal $\epsilon^*$ should be treated as a hyper-parameter, and to be set by the user with try and error.

\begin{figure}
  	\centering
   \begin{tabular}{cc}
   \includegraphics[scale=0.5]{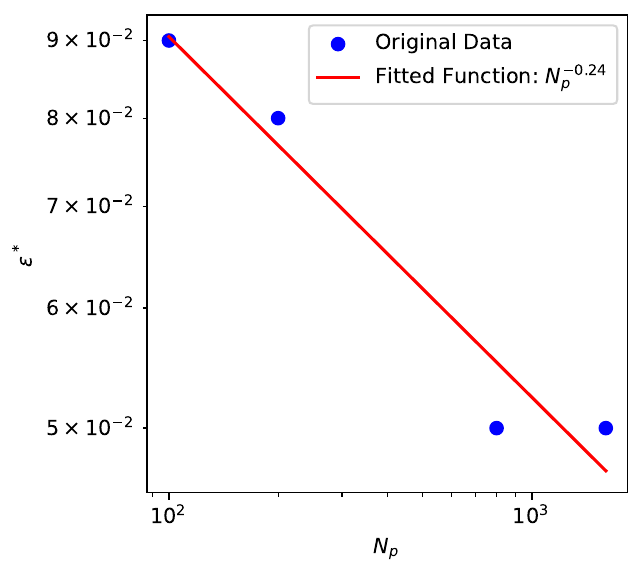}
   &
   \includegraphics[scale=0.5]{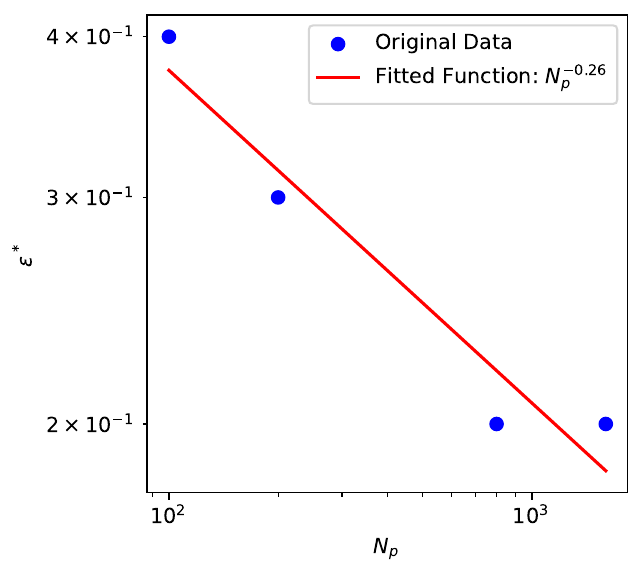}
   \end{tabular}
   \caption{Optimal cutoff denoted by $\epsilon^*$ as a function of $N_p$, given two marginals $\mu=\mathcal{N}( 0, I)$ and $\nu=\mathcal{N}( 1, I)$ over 1-dimensional space, i.e. $\dim(X)=1$ (left) and 2-dimensional space, i.e. $\dim(X)=2$ (right)~. The piecewise constant approximation is employed to estimate the conditional expectation.}
   \label{fig:opt_eps_vs_Np}
\end{figure}

\begin{figure}
  	\centering
   \begin{tabular}{cc}
      \includegraphics[scale=0.5]{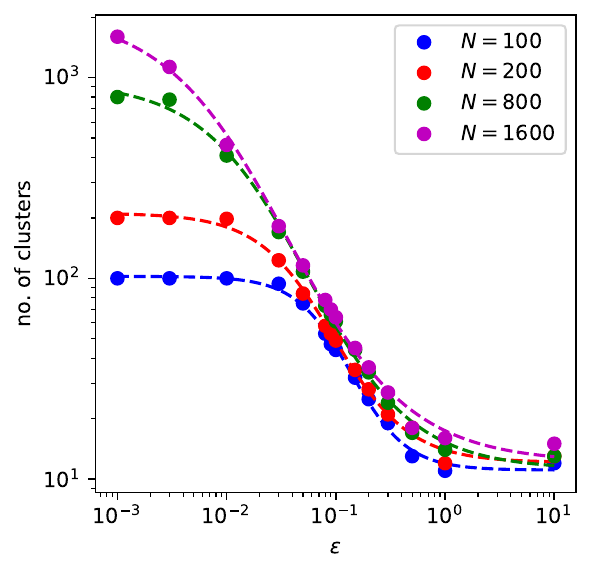}
      &
   \includegraphics[scale=0.5]{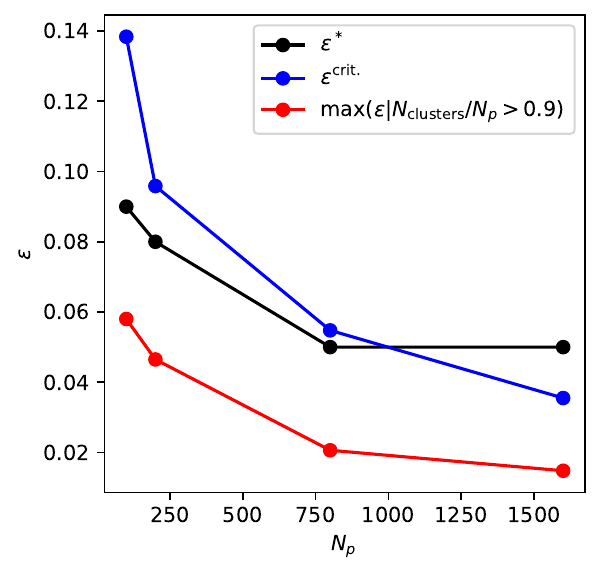}
\\
      \includegraphics[scale=0.5]{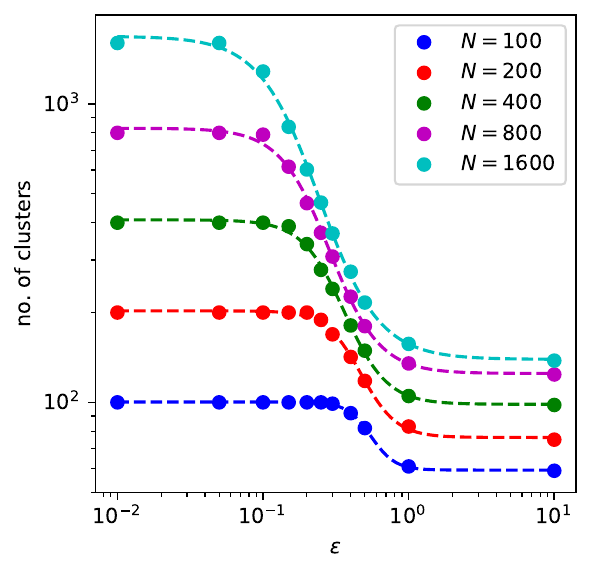}
      &
   \includegraphics[scale=0.5]{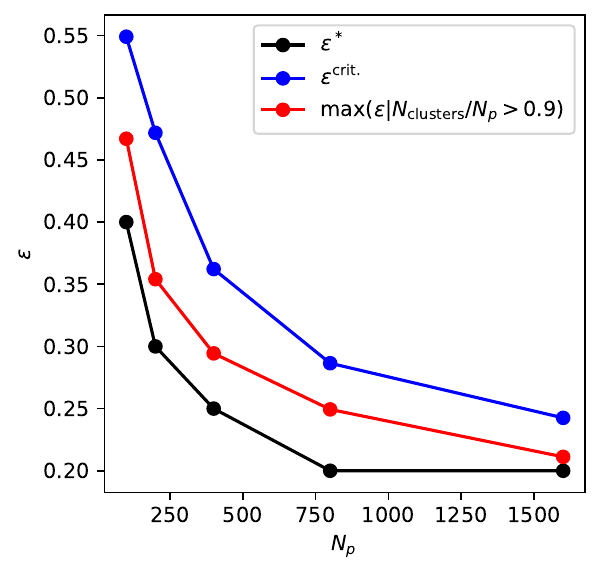}
   \end{tabular}
   \caption{Comparing the critical point of the number of cluster versus $\epsilon$ in 1-dimensional space (top) and two-dimensional space (bottom), where $\mu=\mathcal{N}( 0, I)$ and $\nu=\mathcal{N}( 1, I)$. The piecewise constant approximation is employed for the estimation of conditional expectation.}
   \label{fig:crit_max_versus_opteps}
\end{figure}

\begin{figure}
  	\centering
   \begin{tabular}{cc}
      \includegraphics[scale=0.5]{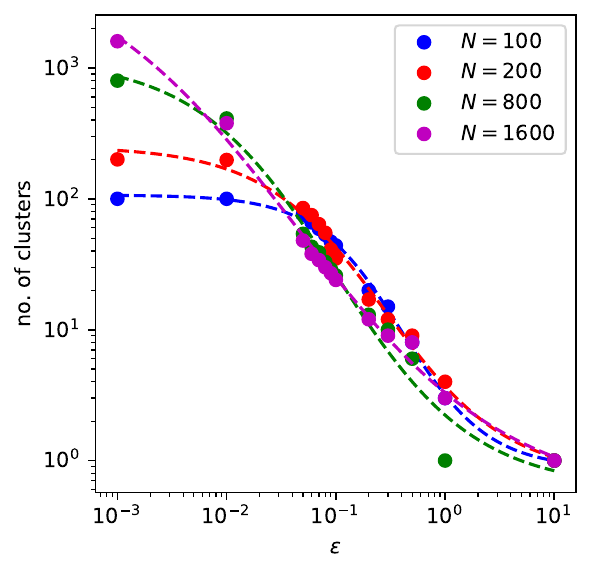}
      &
   \includegraphics[scale=0.5]{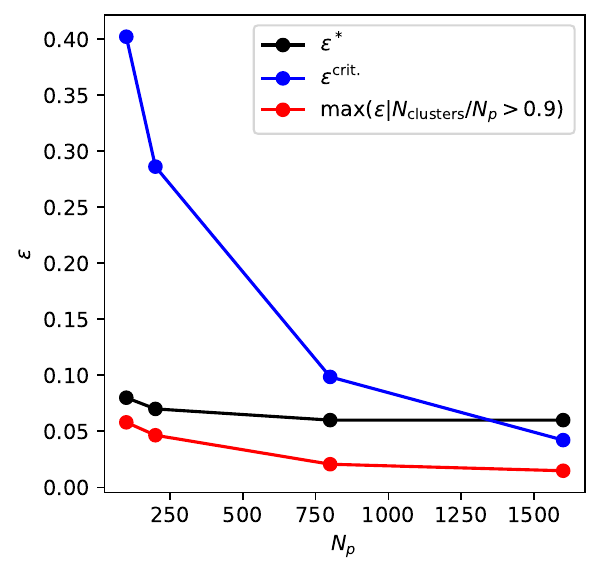}
   \\
   \includegraphics[scale=0.5]{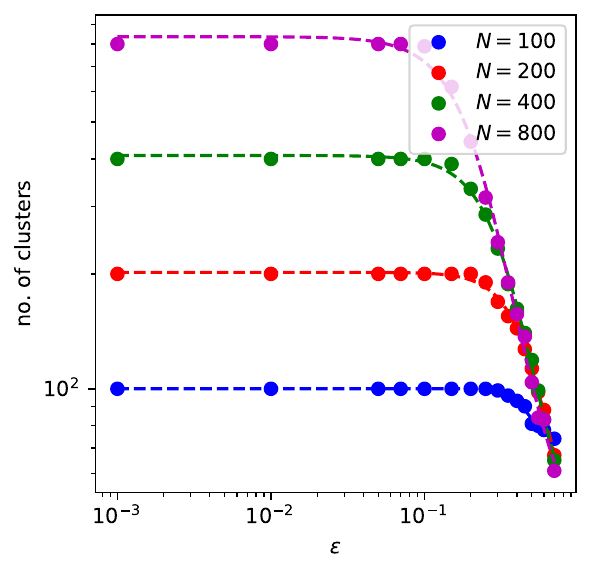}
      &
   \includegraphics[scale=0.5]{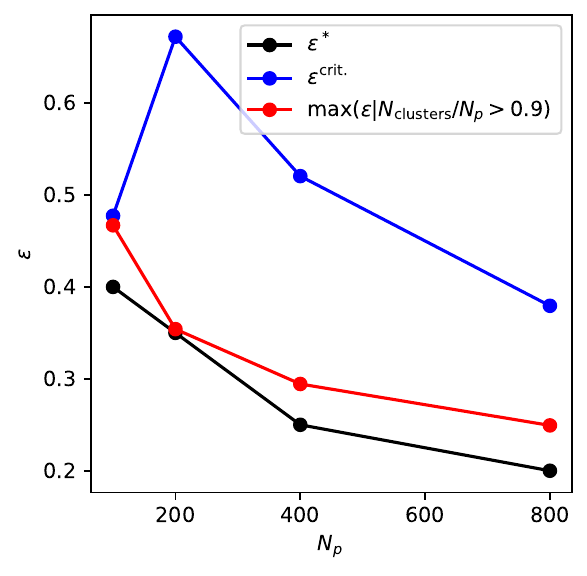}
   \end{tabular}
   \caption{Comparing the critical point of the number of cluster versus $\epsilon$ in 1-dimensional space (top) and 2-dimensional space (bottom), where $\mu=\mathcal{N}( 0, I)$ and $\nu=\mathcal{N}( 1, I)$.  The piecewise linear approximation is considered for the estimation of conditional expectation.}
   \label{fig:crit_max_versus_opteps_lin}
\end{figure}

\section{Gallery of Density Estimation}
\label{sec:further_results}
\noindent To further assess the quality, robustness, and computational cost of OCD, with respect to the state-of-the-art, we compare the learned map between Gaussian and a target distribution using OCD (with $L^2$-cost) and ATM. As target distributions, we consider Swiss roll, banana, and funnel, as shown in Figs.~\ref{fig:test_swiss_roll}-\ref{fig:test_funnel}. In all these experiments, we varied the hyper-parameters of OCD and ATM, and checked how the learned map is affected. We employed $10^4$ samples of the target density in the learning phase, and generated $10^6$ samples of the fitted density to produce the histograms. We solve the OCD equations using RK4 solver and use piecewise linear approximation for the conditional expectation, see \cref{alg:RK4_OCD}. In case of ATM, we used an additional of $500$ samples for validation and performed $5$ folded search. In \cref{table:cost_swiss_roll}, we report the computational cost of OCD and ATM for the Swiss roll example. For reference, the computational cost of EDM is also reported.
\begin{figure}
  	\centering  
    {\tiny

\begin{tabular}{ccc}
& \includegraphics[scale=0.45]{Figures/SwissRoll/Y_swissroll_1000000_noise0.5.pdf}
\\
& Exact
\\
\includegraphics[scale=0.45]{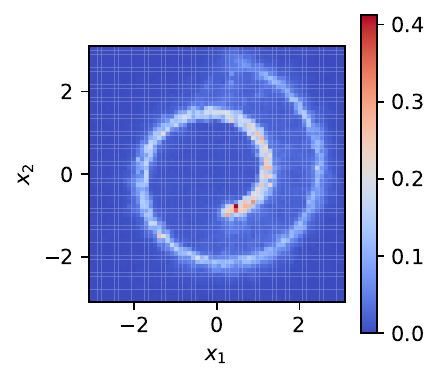}
&   \includegraphics[scale=0.45]{Figures/SwissRoll/OCDRK4_lin_eps0.03_outputY_SwissGauss_10000_dt0.1_noise0.5.pdf}
&
\includegraphics[scale=0.45]{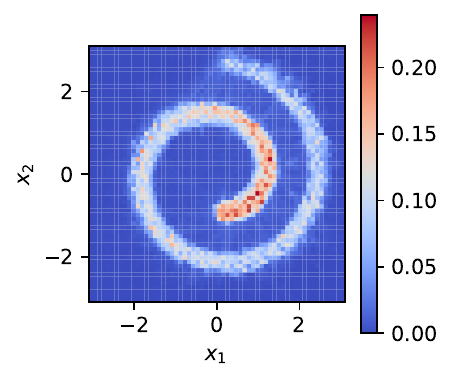}
\\
(a) OCD: $\epsilon=0.025, \Delta t=0.1$
&
(b) OCD: $\epsilon=0.03, \Delta t=0.1$  
&
(c) OCD: $\epsilon=0.035, \Delta t=0.1$
\\
\includegraphics[scale=0.45]{Figures/SwissRoll/OCDRK4_lin_eps0.03_outputY_SwissGauss_10000_dt0.1_noise0.5.pdf}
&
\includegraphics[scale=0.45]{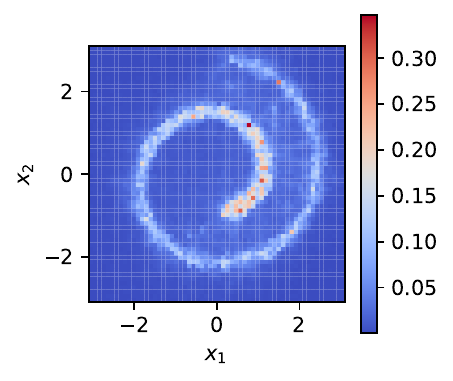}
&
\includegraphics[scale=0.45]{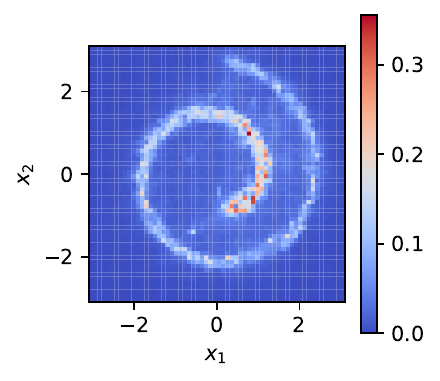}
\\
(d) OCD: $\epsilon=0.03, \Delta t=0.1$
&
(e) OCD: $\epsilon=0.03, \Delta t=0.01$  
&
(f) OCD: $\epsilon=0.03, \Delta t=0.001$
\\
\includegraphics[scale=0.45]{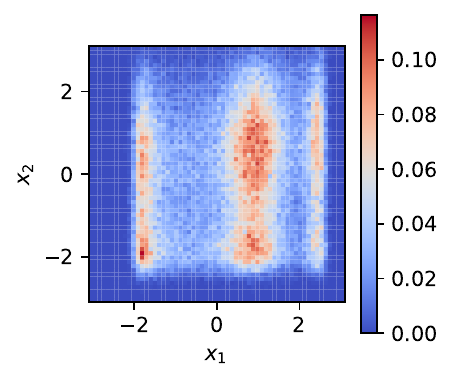}
&
\includegraphics[scale=0.45]{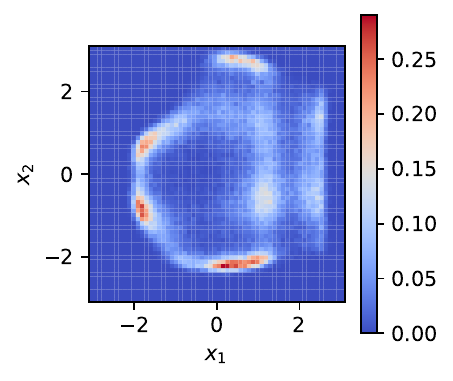}
&
\includegraphics[scale=0.45]{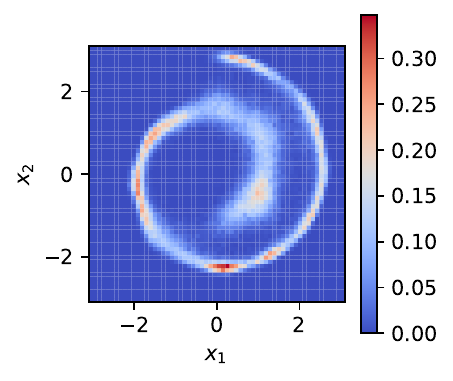}
\\
 (g) ATM (10 terms) & (h)  ATM (20 terms) &  (i) ATM (30 terms)
   \end{tabular}
   }
   \caption{Estimating distribution of  the Swiss roll as the target density from Gaussian source using $L^2$-OCD (a-f) and ATM (g-i).}
   \label{fig:test_swiss_roll}
\end{figure}

\begin{table}
\begin{tabular}{lccc}
 \textbf{Method} & \textbf{Hyper parameters} & \textbf{Ex. time [s]} & \textbf{Memory [MB]}  \\
 \hline
EMD &   & 22.5 &  2400.75\\
\hline
OCD &  $\epsilon=0.025$,  $\Delta t=0.1$ & $91.6$ & $30.9$   \\
& $\epsilon=0.03$, $\Delta t=0.1$ & $84.4$ & $37.1$  
\\
& $\epsilon=0.035$, $\Delta t=0.1$ &  $84.5$ & $37.2$
\\
& $\epsilon=0.03$, $\Delta t=0.01$ &  $743.2$ & $36.2$
\\
& $\epsilon=0.03$, $\Delta t=0.001$ &  $8421.2$ & $40.4$
\\
\hline
ATM & 10 terms &  $31.29$ & $-^*$
\\
ATM & 20 terms &  $207.3$ & $-^*$
\\
ATM & 30 terms & $3542.1$ & $-^*$
\end{tabular}
\caption{Execution time of EMD, OCD, and ATM, along with memory consumption of EMD and OCD  in estimating Swiss roll distribution. \\
${\ }^*$ The deployed code of ATM (\url{https://github.com/baptistar/ATM}) is implemented in Matlab, and we were not able to record the memory consumption accurately on our local computer which uses Linux as the operating system.}
\label{table:cost_swiss_roll}
\end{table}
\begin{figure}
  	\centering
    {\tiny
   \begin{tabular}{ccc}
   & \includegraphics[scale=0.45]{Figures/banana/exactdensity_banana.pdf}
   \\
   & Exact 
   \\
   \includegraphics[scale=0.45]{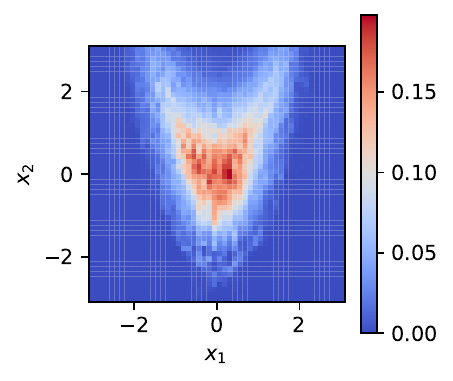}
   &
\includegraphics[scale=0.45]{Figures/banana/OCDRK4_lin_eps0.0656_outputY_banana_20000_dt0.1.pdf}
& 
\includegraphics[scale=0.45]{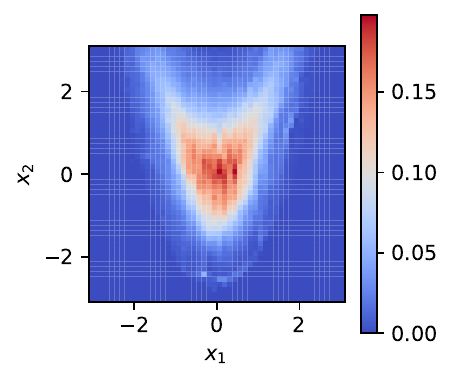}
\\
(a) OCD: $\eps=0.0536$, $\Delta t=0.1$
&
(b) OCD: $\eps=0.0656$, $\Delta t=0.1$
 &
(c) OCD: $\eps=0.0695$, $\Delta t=0.1$
 \\
\includegraphics[scale=0.45]{Figures/banana/OCDRK4_lin_eps0.0656_outputY_banana_20000_dt0.1.pdf} & \includegraphics[scale=0.45]{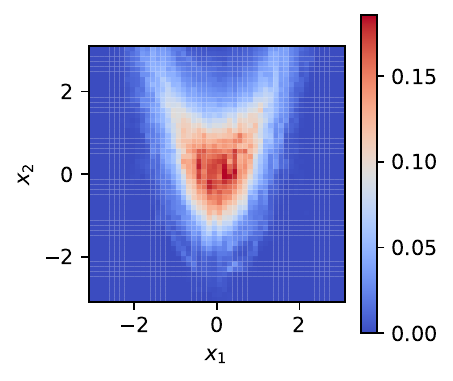}
& 
\includegraphics[scale=0.45]{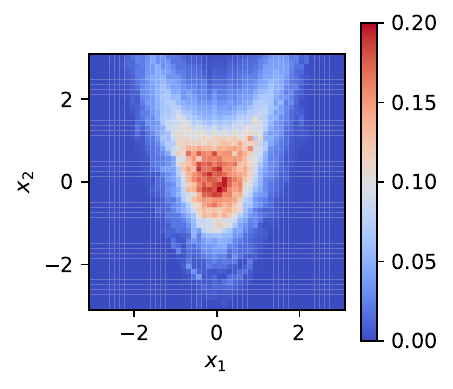}
\\
(d) OCD: $\eps=0.0656$, $\Delta t=0.1$
&
(e) OCD: $\eps=0.0656$, $\Delta t=0.01$
 &
(f) OCD: $\eps=0.0656$, $\Delta t=0.001$
 \\
\includegraphics[scale=0.45]{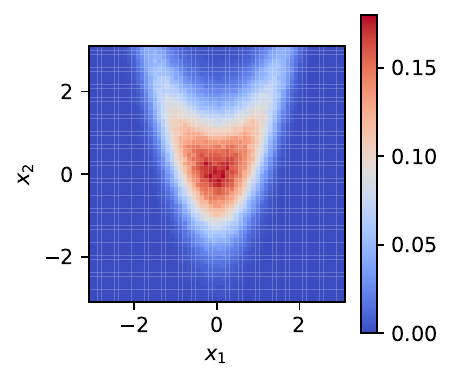}
&
\includegraphics[scale=0.45]{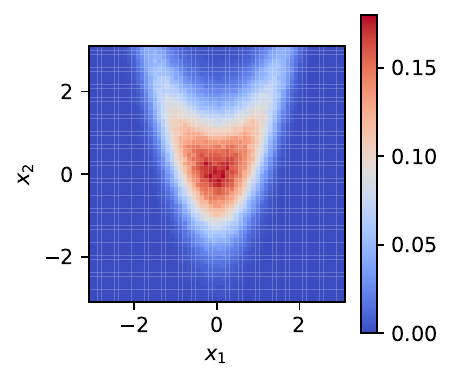}
&
\includegraphics[scale=0.45]{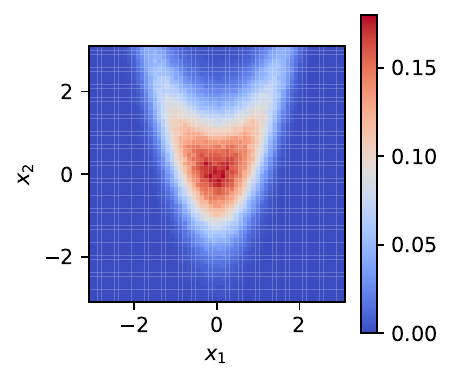}
\\
(g) ATM (10 terms)
&
(h) ATM (20 terms)
&
(i) ATM (30 terms)
   \end{tabular}
   }
   \caption{Estimating distribution of  the banana as the target density from Gaussian source, using $L^2$-OCD (a-f) and ATM (g-i).}
   \label{fig:test_banana}
\end{figure}

\begin{figure}
  	\centering
    {\tiny
   \begin{tabular}{cccc}

& \includegraphics[scale=0.45]{Figures/Funnel/exactdensity_funnel.pdf}
   \\
   & Exact
   
\\
   \includegraphics[scale=0.45]{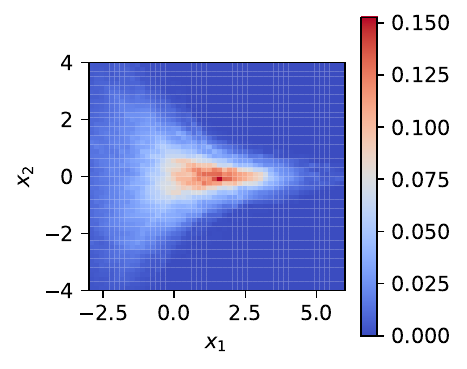}
   &
\includegraphics[scale=0.45]{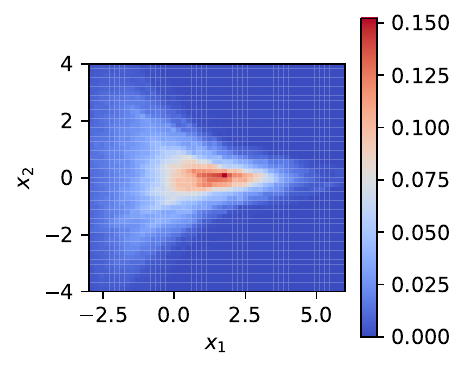}
    &
    \includegraphics[scale=0.45]{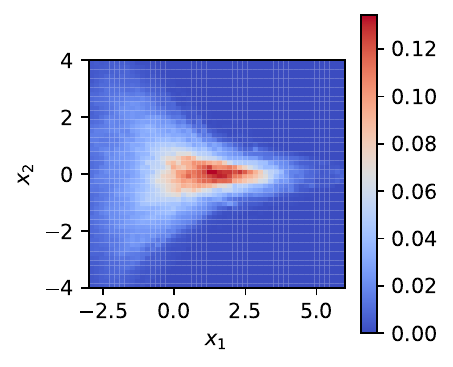}
   \\
    OCD: $\eps=0.088$ $\Delta t= 0.1$
    &
    OCD: $\eps=0.09$ $\Delta t= 0.1$
    &
     OCD: $\eps=0.092$ $\Delta t= 0.1$
    \\
    \includegraphics[scale=0.45]{Figures/Funnel/OCDRK4_lin_eps0.09_outputY_funnel_10000_dt0.1.pdf}
    &
    \includegraphics[scale=0.45]{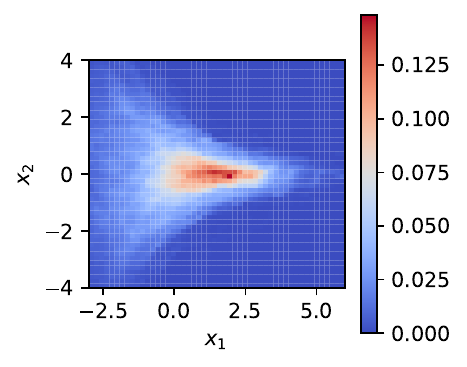}
    &
    \includegraphics[scale=0.45]{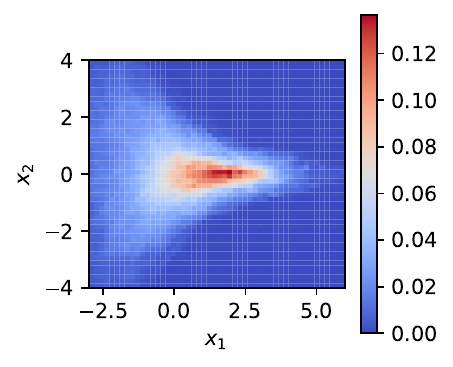}
    \\
OCD: $\eps=0.09$, $\Delta t=0.1$
&
 OCD: $\eps= 0.09$, $\Delta t=0.01$
 &
 OCD: $\eps=0.09$, $\Delta t=0.001$
 \\
\includegraphics[scale=0.45]{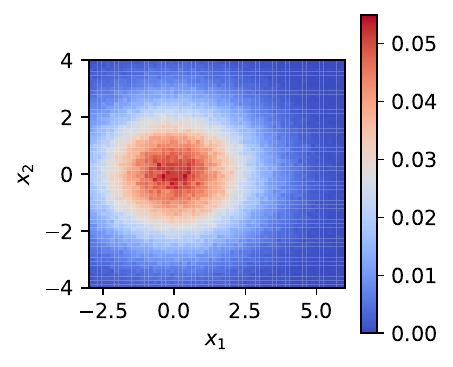}
& 
\includegraphics[scale=0.45]{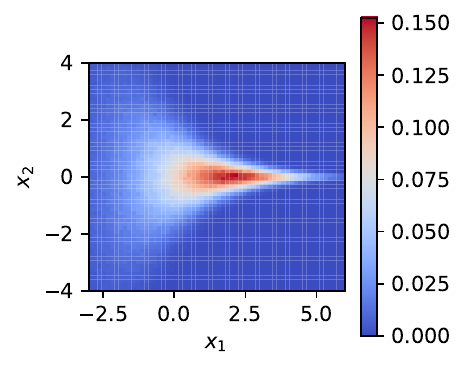}
&
\includegraphics[scale=0.45]{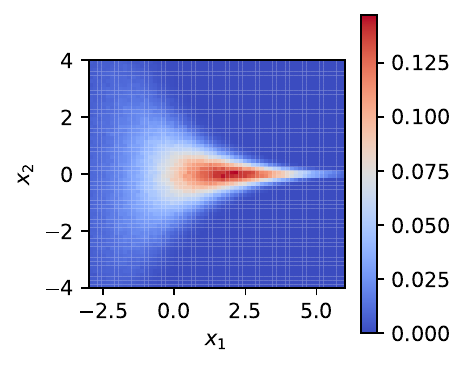}
\\
ATM (10 terms)
&
ATM (15 terms)
&
ATM (20 terms)
   \end{tabular}
   }
\caption{Estimating distribution of  the funnel as the target density from Gaussian source, using $L^2$-OCD (a-f) and ATM (g-i).}
   \label{fig:test_funnel}
\end{figure}


\section{Benchmarking OCD against Neural Optimal Transport}
\label{sec:benchmark_NOT}
In this section, we further test OCD against another particle method called Neural Optimal Transport \cite{korotin2023neural}. In this method, backpropagation is used to compute the map between marginals. We consider two normal distributions $X,Y\sim N(0_n,I_n)$ in $n$-dimensional setting. We compare the solution of the discrete samples from each method against EMD and investigate performance in terms of execution time and error.  As shown in Fig.~\ref{fig:OCD_vs_NOT}, for a fixed dimension of $n=2$, using 200 epochs for NOT and 100 iterations for OCD, the proposed OCD method outperforms NOT for a wide of number of particles $N_p$, while achieving more accurate solution. However, we expect NOT to reach the performance of OCD at very large $N_p\rightarrow \infty$. In Fig.~\ref{fig:OCD_vs_NOT_n}, we further compare NOT to OCD in terms of dimension $n=\dim(X)$. Clearly, the OCD outperforms NOT in low dimensional problems. We expect a turnover in performance in high dimensions.

\begin{figure}[H]
  	\centering
   \begin{tabular}{cc}
   \includegraphics[scale=0.65]{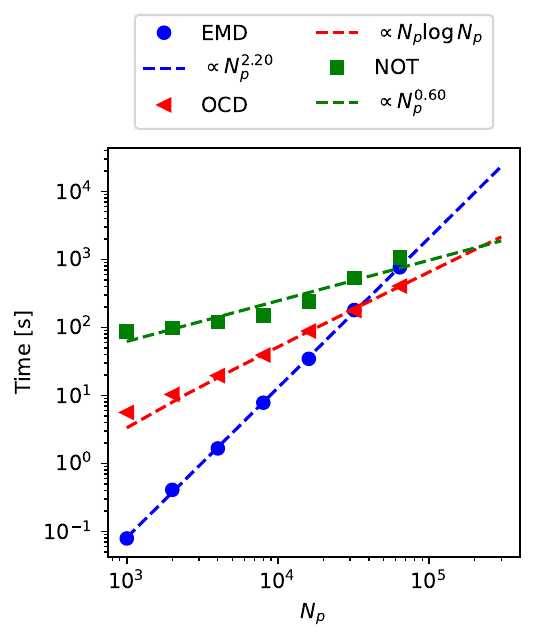}
   &
   \includegraphics[scale=0.65]{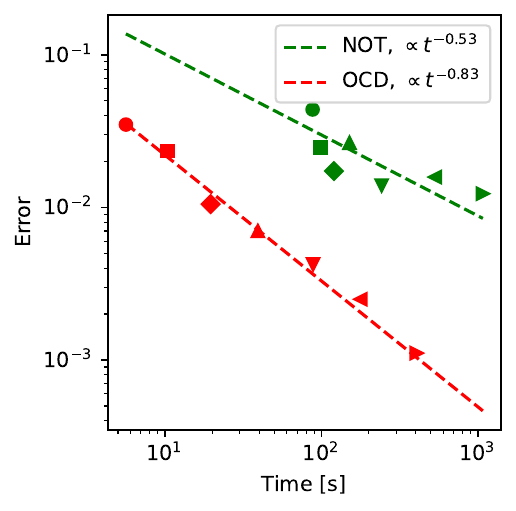}
   \end{tabular}
   \caption{Scaling of execution time against number of particles $N_p$ (left) and a comparison of NOT against OCD in terms of $L_1$ error in the  estimation of Wasserstein distance versus computational time (right).}
   \label{fig:OCD_vs_NOT}
\end{figure}

\begin{figure}[H]
  	\centering
   \begin{tabular}{cc}
   \includegraphics[scale=0.65]{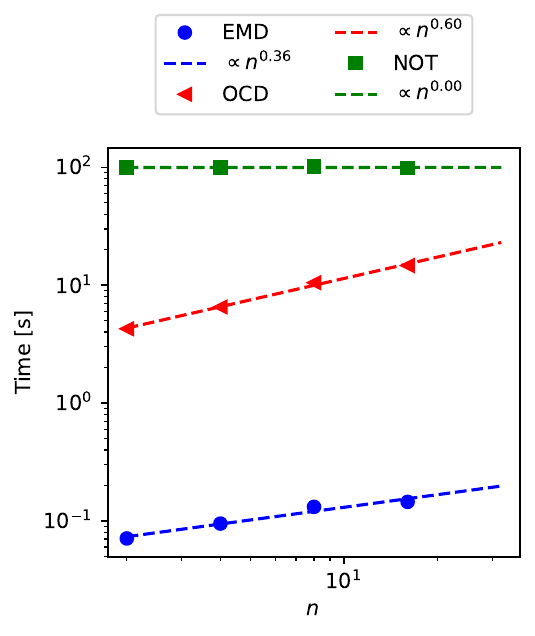}
   &
   \includegraphics[scale=0.65]{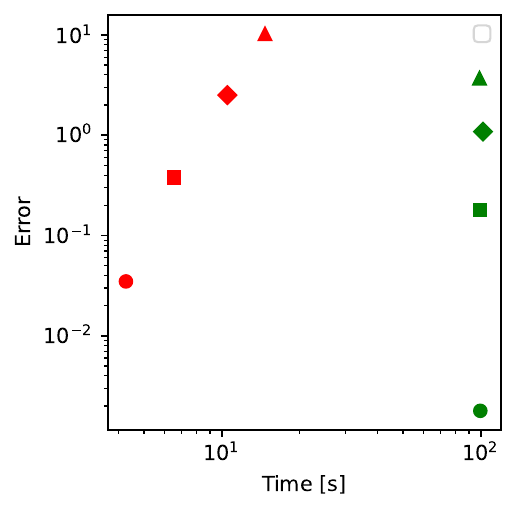}
   \end{tabular}
   \caption{Scaling of execution time against dimension $n=\dim(X)$ (left) and a comparison of NOT against OCD in terms of $L_1$ error in the  estimation of Wasserstein distance versus computational time for the corresponding dimensions (right).}
   \label{fig:OCD_vs_NOT_n}
\end{figure}


\section{OCD with $L^p$-cost function}
Here, we demonstrate the application of the developed $L^p$-OCD method for convex cost functions other than quadratic one considered in the main text. As an example, let us consider $L^p$ function $c_p(x,y)=||x-y||_p^p$ as the cost function for the optimal map problem with value $p=4$. 
\\ \ \\
As the first example, we consider $n=1,2$ and generate $10^4$ samples from $X\sim \mathcal{N}(0,I)$ and $Y\sim \mathcal{N}(0,0.5I)$. As shown in Fig.~\ref{fig:OCD_other_cost}, both of these cases converge to the sub-optimal solution monotonically in time.

\begin{figure}[H]
  	\centering
   \begin{tabular}{cc}
   \includegraphics[scale=0.65]{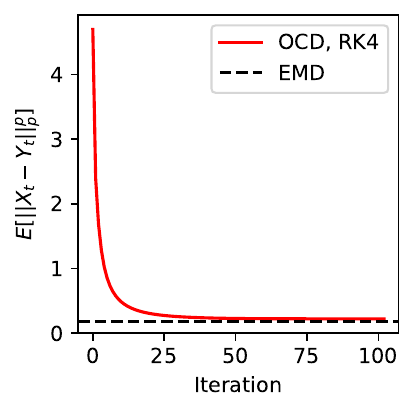}
   &
\includegraphics[scale=0.65]{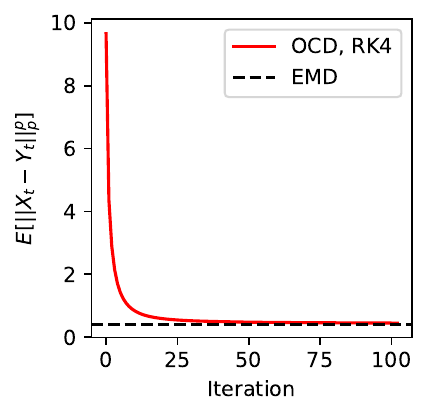}
\\
$n=1$ & $n=2$
   \end{tabular}
   \caption{Estimate of $L^p$-Wasserstein distance in the OCD dynamics with $p=4$ till steady-state between two normally distributed random variables $X$ and $Y$ with $\Delta t=0.01$.}
\label{fig:OCD_other_cost}
\end{figure}

\noindent \sloppy Second, we consider 
the optimal map problem between the normal distribution $X\sim N(0,I)$ and its push-forward softmax map
$Y=T(X)$, where 
$T_i=\nabla_{x_i} \log \left(\exp(x_1)+\exp(x_2)\right)$ in $\mathbb R^2$. Starting from $10^5$ independent random samples of both marginals, we evolve the pair $(X_t,Y_t)$ by RK4-OCD with the linear approximation to the conditional expectation, $\Delta t=3\times 10^{-3}$ and $\epsilon=2\times10^{-3}$. As shown in Fig.~\ref{fig:OCD_lp_softmax}, $\mathbb{E}\lVert X_t-Y_t \rVert_p^p$ converges towards the  $L^p$-Wasserstein distance.

\begin{figure}[H]
  	\centering
   \begin{tabular}{c}
   \includegraphics[scale=0.65]{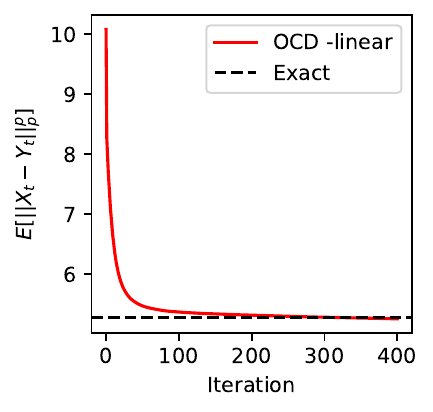}
   \end{tabular}
   \caption{Estimate of $L^p$-Wasserstein distance in the OCD dynamics with $p=4$ till steady-state between normally distribution $X$ and softmax map $Y$ with $\epsilon=2\times 10^{-3}$ and $\Delta t=3 \times 10^{-3}$.}
\label{fig:OCD_lp_softmax}
\end{figure}

\section{$10$-dimensional multivariate normal distributions}

As an example in the high-dimensional setting, here we test the proposed OCD in estimating the optimal map between two multivariate normal distributions $\mu=\mathcal{N}(0,\Sigma_\mu)$ and $\nu=\mathcal{N}(0,\Sigma_\nu)$ in $\mathbb R^n$ with $n=10$ where
\begin{flalign}
&\Sigma_\mu = I + L_\mu^TL_\mu
\ \ \  
\Sigma_\nu = I+ L_\nu^TL_\nu
\label{eq:cov_mult-variate}
\\
&(L_\mu)_{i,j} = 
\left\{\begin{matrix}
\frac{10}{(i+j)^2} \ \ \ \forall j < i  \\
0 \ \   \textrm{otherwise}
\end{matrix}\right.
\\
&(L_\nu)_{i,j} = 
\left\{\begin{matrix}
\cos\left(2\pi(i+j)/n\right) \ \ \ \forall j < i
 \\
0 \ \   \textrm{otherwise}
\end{matrix}\right.
\end{flalign}
We generate $10^4$ independent samples for $\mu$ and $\nu$ and estimate the optimal map between them using OCD with $\epsilon=1.88$ and $\Delta t = 0.05$  against the solution obtained by EMD. As shown in Fig.~\ref{fig:OCD_10dim_multivariate}, OCD provides an accurate estimate for the $L^2$ Wasserstein distance.

\begin{figure}[H]
  	\centering
   \begin{tabular}{c}
   \includegraphics[scale=0.65]{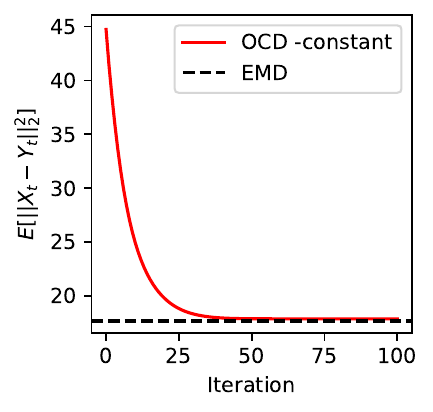}
   \end{tabular}
   \caption{Estimate of $L^2$-Wasserstein distance in the OCD dynamics with piecewise constant estimator of the conditional expectation till steady-state between multi-variate 10-dimensional normal distributions $\mu$ and $\nu$ with covariance matrices \eqref{eq:cov_mult-variate} against the solution obtained with EMD.}
\label{fig:OCD_10dim_multivariate}
\end{figure}


\end{document}